\DeclareSymbolFont{AMSb}{U}{msb}{m}{n} 			
\documentclass[12pt,noamsfonts,a4paper]{amsart}

\usepackage{enumerate}

\usepackage[utf8]{inputenx} 
\usepackage[T1]{fontenc}

\usepackage{microtype} 

\usepackage{bera} 
\usepackage[bitstream-charter,cal=cmcal]{mathdesign} 
\renewcommand{\in}{\smallin} \renewcommand{\notin}{\notsmallin} \renewcommand{\setminus}{\smallsetminus} 

\newcommand{\om}{\omega}

\DeclareMathOperator{\cof}{{cof}}
\DeclareMathOperator{\cov}{{cov}}

\DeclareMathOperator{\non}{{non}}

\newcommand{\sub}{\subseteq}

\newcommand{\fin}{\textup{Fin}}

\newcommand{\lxp}[2]{{\vphantom{#2}}^{#1}{#2}}
\newcommand{\Pw}[1]{\mathcal{P}(#1)}
\newcommand{\pom}{\Pw{\om}}
\newcommand{\pofin}{\pom / \fin}
\newcommand{\fc}{\lxp{\om\!\!}{\om}}

\newtheorem{thm}{Theorem}
\newtheorem{lem}[thm]{Lemma}
\newtheorem*{thm*}{Theorem}
\newtheorem*{lem*}{Lemma}
\newtheorem{cor}[thm]{Corollary}
\newtheorem{prop}[thm]{Proposition}
\newtheorem{fact}[thm]{Fact}

\theoremstyle{definition}
\newtheorem{defi}[thm]{Definition}

\newtheorem{prob}[thm]{Problem}
\newtheorem{question}[thm]{Question}

\theoremstyle{remark}
\newtheorem{remark}[thm]{Remark}
\newtheorem*{remark*}{Remark}

\newtheorem{example}[thm]{Example}
\newtheorem*{claim}{Claim}

\begin{document}

\title{Hausdorff Gaps and Towers in $\pofin$}

\thanks{The first author was partially supported by the grant N N201 418939 from the Polish Ministry
of Science and Higher Education.
The second author was partially supported by the grant IAA100190902 of GA AV \v{C}R and  RVO: 67985840}
\date{\today} 

\author[Piotr Borodulin-Nadzieja]{Piotr Borodulin-Nadzieja}
\address[Piotr Borodulin-Nadzieja]{Instytut Matematyczny \\ Uniwersytet Wroc\l awski \\ pl.\,Grunwaldzki 2/4 \\ 50-384 Wroc\l aw \\ Poland}
\email{pborod@math.uni.wroc.pl}

\author[David Chodounsk\'{y}]{David Chodounsk\'{y}}
\address[David Chodounsk\'{y}]{Institute of Mathematics AS CR \\ \v{Z}itn\'{a} 25 \\ 115 67 Praha 1 \\ Czech Republic  \&
Department of Mathematics \\ York University \\ 4700 Keele Street \\ 4700 Keele Street \\ Toronto \\ Ontario M3J 1P3 \\ Canada}
\email{david.chodounsky@matfyz.cz}
\thanks{Corresponding author: David Chodounsk\'{y}}

\begin{abstract} 
We define and study two classes of uncountable $\sub^*$-chains: Hausdorff towers and Suslin towers. We discuss their existence in various models of set theory. Some of the results and methods are used to provide examples of indestructible gaps not equivalent to a Hausdorff gap. We also indicate possible ways of developing a structure theory for towers based on classification of their Tukey types.
\end{abstract}

\maketitle
\section{Introduction}

We say that subsets $A$, $B$ of $\omega$ are in the relation of almost inclusion (denoted by $A\sub^* B$) if $A\setminus B$ is finite.
One of the motivations of this article is the following question:
\begin{question} \label{motivation}
Is there an uncountable well-ordered $\sub^*$-chain which consists of pairwise $\sub$-incomparable elements?
\end{question}
In a sense this is the question how ``far'' is $\sub^*$ from $\sub$.

The answer to Question \ref{motivation} is positive. We will call well-ordered increasing $\sub^*$-chains towers. (We do not assume that towers are maximal with respect to end-extension as is often done in the literature, but we treat here only uncountable towers.) 
There are both towers witnessing the positive answer to Question \ref{motivation} (we will call them \emph{special}) and towers which do not have an uncountable subtower consisting of $\sub$-incomparable sets (called \emph{Suslin}). Examples of both sorts are implicitly mentioned in
\cite{oscillations}. 

A tower $(T_\alpha)_{\alpha<\om_1}$ satisfies condition (H) if the set $\{\xi<\alpha\colon T_\xi \setminus T_\alpha \sub n\}$ is finite for each $\alpha<\omega_1$ and $n<\om$. Although it seems that this notion has not appeared
explicitly in the literature, the reader can recognize a resemblance between condition (H) and the well-known Hausdorff condition for gaps (see Section \ref{gaps}). This is not a coincidence: every ``left half'' of a Hausdorff gap is a Hausdorff tower,
i.e.\ a tower containing a cofinal subtower which satisfies condition (H). It turns out, that Hausdorff towers are the natural examples of special towers. Moreover, by an easy modification of arguments used for analyzing gaps, one can show that under
$\mathsf{MA}_{\om_1}$ all towers of length $\om_1$ are Hausdorff. So despite the fact that the object as in Question \ref{motivation} could seem unusual at first glance, it is quite common. In Section \ref{special-towers} we discuss models in which
all $\om_1$ towers are special. Moreover, we show that for every $\kappa$-tower (where $\kappa>\omega$ is regular) there is a ccc forcing making it special in the extension. In fact, under $\mathsf{MA}_{\kappa}$ each $\kappa$-tower is very close to
be a $\subseteq$-antichain of size $\kappa$ (Theorem \ref{MA-similar}).
An analysis of the analogous Luzin condition for almost disjoint systems in $\Pw{\omega}$ was done by Guzm\'{a}n and Hru\v{s}\'{a}k. 
Not surprisingly, many results about Hausdorff towers and Luzin gaps are in a direct correspondence, see~\cite{Michael-Osvaldo}. 

To the best of our knowledge, the first example of a tower which does not contain an uncountable $\sub$-antichain was given in \cite{KunenDouwen} under the assumption of $\mathsf{CH}$. More examples are provided by results from \cite{oscillations}. 
Todor\v{c}evi{\'c} proved there a theorem (see Theorem \ref{osci} in Section \ref{suslin-towers}) which implies that every tower of uncountable cofinality generating a non-meager ideal is Suslin. I.e., every tower rich enough (e.g.\ generating a
maximal ideal) cannot be special. There are also Suslin towers generating meager ideals, see Section \ref{suslin-towers}.

The analogy between towers and gaps is strong, at least in the sense that many results 
about gaps can be easily modified for the case of towers. E.g.\ under $\mathsf{MA}_{\om_1}$ each gap is Hausdorff as well as each tower
(of size $\omega_1$) contains a subtower with condition (H). In a model obtained by adding a single Cohen real we can produce a non-special gap and a non-special tower practically in the same way. However, this analogy breaks up in many ways. 
Under $\mathsf{PID}$ each gap is Hausdorff, but we show that the existence of a non-special tower is consistent with $\mathsf{PID}$ (see Section~\ref{special-towers}). 
On the other hand, Theorem~\ref{PID} states that $\mathsf{PID} + \om_1 < \mathfrak b$ is a sufficient condition for all towers being Hausdorff.
It becomes apparent in Section~\ref{Tukey} that this result is related to the ``only 5 Tukey types'' theorems.
We  also prove that consistently there is a Hausdorff tower which generates a dense ideal and thus cannot be a half of any gap (Example~\ref{gen-haus}), and a special tower which is equivalent (in the sense
of generating the same ideal) to a Suslin tower and thus is not Hausdorff (Example~\ref{special-nonHaus}). 
Some of these results are contained implicitly in~\cite{oscillations}. 

The theory of towers is a debtor of the theory of gaps, but it is not an ungrateful one. In fact, the analysis of the property of being a special tower has led us
to an example of a gap which is special but not equivalent to a Hausdorff gap (Example \ref{special-nonLO}). In \cite{scheepers} Scheepers asked about the existence of such an object and Hirschorn in \cite{hirschorn} answered this question affirmatively. 
Our example is of a different sort than the one of Hirschorn and it has a simpler description. Namely, Hirschorn showed that there is a special gap which does not satisfy a certain weaker condition than being Hausdorff (we call it \emph{left-oriented}). We present an
example which is left-oriented but not Hausdorff. In Section \ref{gaps-towers} we offer other examples of this kind (many of them exist in any model obtained by adding $\om_1$ many Cohen reals). In particular, we prove the consistent existence of
a Hausdorff gap $(L_\alpha, R_\alpha)_{\alpha<\omega_1}$ such that $(R_\alpha, L_\alpha)_{\alpha<\omega_1}$ is not Hausdorff (Example \ref{uninvertible-haus}), a special gap $(L_\alpha,R_\alpha)_{\alpha<\omega_1}$ such that neither $(L_\alpha,R_\alpha)_{\alpha<\omega_1}$
nor $(R_\alpha,L_\alpha)_{\alpha<\omega_1}$ are left-oriented (Example \ref{ex-non-oriented}), and a gap $(L_\alpha,R_\alpha)_{\alpha<\omega_1}$ which is left-oriented but not Hausdorff and $(R_\alpha,L_\alpha)_{\alpha\in\omega_1}$ is special but not left-oriented
(Theorem \ref{David0}). At the end of Section \ref{gaps-towers} we come back to towers to construct a special non-Hausdorff tower which is not equivalent to a Suslin tower.

Towers are often used as a combinatorial tool in set theory, set theoretic topology and functional analysis. E.g.\ Stone spaces of Boolean subalgebras of $\mathcal{P}(\omega)$ generated by towers (and $[\omega]^{<\omega}$) are ordered compacta being continuous images of $\om^*$. Bell in \cite{Bell} used a tower to construct a compact separable space which does not continuously map onto $[0,1]^{\om_1}$ and which does not have a countable $\piup$-base. In \cite{KunenDouwen} a non-special tower generates an L-space and a S-space, both  subspaces of $\mathcal{P}(\om)$ equipped with the Vietoris topology. However, no additional properties of towers are usually needed (with the exception of the last result),
except possibly of some maximality properties like generating a dense ideal (i.e. such that every infinite subset of $\omega$ contains an infinite element of the ideal), or a maximal ideal. Perhaps this is the reason why there were not many attempts to develop a structure theory for towers. 

This article can be treated as a modest contribution to the program of filling this gap. Properties of being special or Hausdorff demarcate some dividing lines in the class of towers. In Section \ref{Tukey} we try to examine possible ways to expand
this research.
We use the Tukey ordering, a tool which has proved its worth in exploring the structure of ultrafilters (see \cite{Natasha}). We show that an $\omega_1$-tower is Hausdorff if and only if it is Tukey top among directed sets of size $\omega_1$. Using results from \cite{Natasha}, we observe that consistently there are $2^\mathfrak{c}$ many pairwise incomparable Tukey types of $\omega_1$-towers.

\section{Preliminaries on gaps}\label{gaps}

It will be convenient to start with definitions and basic facts about gaps. 
More details can be found in \cite{scheepers} and \cite{yorioka}.

Recall that $\left(L_\alpha,R_\alpha\right)_{\alpha< \om_1}$ is \emph{a pre-gap} if $L_\alpha \cap R_\alpha = \emptyset$ for each $\alpha<\om_1$ and
both $\left(L_\alpha\right)_{\alpha<\om_1}$ and $\left(R_\alpha\right)_{\alpha<\om_1}$ are towers. A pre-gap $\left(L_\alpha,R_\alpha\right)_{\alpha<\om_1}$ forms \emph{an ($\omega_1,\omega_1$)-gap}
if there is no set $L$ interpolating it, i.e.\ no set $L$ such that $L_\alpha\sub^* L$ and $R_\alpha\cap L =^* \emptyset$ for every $\alpha<\om_1$.

More generally, $\left(L_\alpha, R_\beta\right)_{\alpha<\lambda, \beta<\kappa}$ is a $(\lambda,\kappa)$-gap if $L_\alpha \cap R_\beta =^* \emptyset$ for every $\alpha<\lambda$ and $\beta<\kappa$ and there is no $L$ such that $L_\alpha\sub^* L$ and $L
\subsetneq^* R^c_\beta$ for every $\alpha<\lambda$ and $\beta<\kappa$. Notice that the last inequality is slightly more complicated than the equality $L \cap R_\beta =^* \emptyset$ but this setting enables us to consider $(\lambda,1)$-gaps. In what follows, a gap is an $(\omega_1,\omega_1)$-gap unless stated otherwise.

We say that a gap $(L_\alpha,R_\alpha)_{\alpha<\om_1}$ satisfies \emph{condition (H)} if
\[ \left\{\xi<\alpha\colon L_\xi \cap R_\alpha \sub n\right\} \text{ is finite } \]
for each $\alpha<\om_1$ and $n< \om$. Similarly, a (pre-)gap satisfies \emph{condition (K)} if
\[\left( L_\alpha \cap R_\beta\right) \cup \left(L_\beta \cap R_\alpha\right) \neq \emptyset \]
for each $\alpha < \beta<\om_1$. Finally, a (pre-)gap satisfies \emph{condition (O)} if
\[ L_\alpha \cap R_\beta \neq \emptyset\]
for each $\alpha < \beta<\om_1$.

Now we are ready to define basic types of gaps (the first two are well-known in the literature). 

\begin{defi}\label{gap-types}
	A subgap of a gap $(L_\alpha, R_\alpha)_{\alpha<{\kappa}}$ is a gap $(L_\alpha, R_\alpha)_{\alpha\in I}$, where $I$ is a cofinal subset of $\kappa$.
	A gap is called \emph{Hausdorff} if it contains a subgap satisfying condition (H).
	A gap is called \emph{special} (or \emph{indestructible}) if it contains a subgap satisfying condition~(K).
	A gap $(L_\alpha, R_\alpha)_{\alpha<{\om_1}}$ is called \emph{left-oriented} (or just  \emph{oriented}) if it contains a subgap satisfying condition~(O). It is \emph{right-oriented} if $(R_\alpha, L_\alpha)_{\alpha<{\om_1}}$ is left-oriented.
\end{defi}

The name ``indestructible'' for special gaps is due to the fact that these are precisely gaps indestructible by $\om_1$ preserving forcing notions.

\begin{thm}[Kunen, see \cite{scheepers}]
	For a gap $\mathcal G = \left(L_\alpha, R_\alpha\right)_{\alpha<\omega_1}$ the following are equivalent
	\begin{enumerate}
		\item $\mathcal G$ is special;
		\item $\mathcal G$ is a gap in every $\om_1$ preserving extension of the universe of sets $V$;
		\item $\mathcal G$ is a gap in every generic extension of the universe obtained by a ccc-forcing.
	\end{enumerate}
\end{thm}

For $i < 2$ consider the gaps $\left(L^i_\alpha,R^i_\alpha\right)_{\alpha<\om_1}$.  We say that these two gaps are \emph{equivalent} if $\mathcal{L}^0 = \mathcal{L}^1$  and $\mathcal{R}^0 = \mathcal{R}^1$, where $\mathcal{L}^i$ is the ideal generated by
$\left(L_\alpha^i\right)_{\alpha<\om_1}$ (i.e.\ $\mathcal{L}^i = \left\{A\sub \om\colon \exists \alpha<\om_1 \ A\sub^* L^i_\alpha\right\}$), and $\mathcal{R}^i$ is the ideal generated by $\left(R_\alpha^i\right)_{\alpha<\om_1}$.

\begin{lem}
Properties in Definition \ref{gap-types} respect equivalence of gaps.	
\end{lem}
\begin{proof}
	Suppose $(L_\alpha', R_\alpha')$ satisfies condition ($\star$) (where $\star$ is one of H, K, O) and $(L_\alpha, R_\alpha)$ is an equivalent gap.
	We can find cofinal subgaps 
	\[(L_\alpha', R_\alpha')_{\alpha \in I'} = (O_\alpha', P_\alpha')_{\alpha < \om_1}, \ \ \ \ (L_\alpha, R_\alpha)_{\alpha \in I} = (O_\alpha, P_\alpha)_{\alpha < \om_1},\] 
	and an integer $n$ such that $O_\alpha'\setminus n \sub O_\alpha$, $P_\alpha'\setminus n \sub P_\alpha$, and both $O_\alpha' \cap n$ and $P_\alpha' \cap n$ are constant for each $\alpha < \om_1$.
	Since $(O_\alpha', P_\alpha')_{\alpha < \om_1}$ satisfies ($\star$) and for $\alpha, \beta < \om_1$ $O_\alpha' \cap P_\beta'\sub O_\alpha \cap P_\beta$, the gap $(O_\alpha, P_\alpha)_{\alpha < \om_1}$ satisfies ($\star$) as well.
\end{proof}

The following simple fact reveals the connection between Hausdorff and left oriented gaps. 

\begin{fact} \label{Hausdorff is left oriented}
	Every Hausdorff gap $\mathcal G = (L_\alpha, R_\alpha)$ is a left oriented (special) gap.
\end{fact}
\begin{proof}
	Define a set mapping $f\colon \om_1 \to [\om_1]^{<\om}$ by 
	\[ f(\alpha) = \left\{\xi<\alpha\colon L_\xi \cap R_\alpha = \emptyset\right\}. \]
	Hajnal's free set theorem (see e.g.\ \cite[Corollary 44.2]{erdos}) implies that there is an unbounded $X\sub \om_1$ such that $\xi\notin f(\alpha)$ for each $\xi$, $\alpha\in X$. This means that $L_\xi \cap R_\alpha\ne \emptyset$ for each $\xi < \alpha\in X$.
\end{proof}

Under $\mathsf{MA_{\om_1}}$ or $\mathsf{PID}$ (see \cite{abraham}) every gap is Hausdorff.
It is consistent to have special non-Hausdorff gaps; the first example of such gap was constructed in \cite{hirschorn}.
In Section \ref{gaps-towers} we provide a construction of a special non-Hausdorff gap of a quite different nature.

For a given tower $(T_\alpha)_{\alpha<\om_1}$ is always possible to construct
a Hausdorff 
gap $(L_\alpha,R_\alpha)_{\alpha<\om_1}$ such that
$L_\alpha \cup R_\alpha = T_\alpha$ for each $\alpha < \om_1$. 
It is even possible to construct a large system of such 
gaps \cite{talayco,farah,morgan2}.

It is worth mentioning that there is an analogy between gaps and Aronszajn trees in which destructible gaps correspond to Suslin trees (see \cite[Section 2.2]{abraham}).
Indeed, if for a given pre-gap $ \mathcal{G} = (L_\alpha, R_\alpha)_{\alpha<\omega_1}$ we introduce a compatibility relation on $\om_1$ in the following way: 
$\alpha$, $\beta<\om_1$ are compatible if  \[ \left(L_\alpha \cap R_\beta\right) \cup \left(L_\beta \cap R_\alpha\right) = \emptyset, \]
then $\mathcal{G}$ is a gap if and only if there are no uncountable chains (of pairwise compatible elements) in $\Pw{\om_1}$. Moreover, $\mathcal{G}$ is destructible if and only if there are no uncountable antichains (of~pairwise incompatible elements) in $\Pw{\om_1}$. This remark explains an analogy in results about destructible gaps and Suslin trees. E.g.\ adding a Cohen real adds both a destructible gap and a Suslin tree; under $\mathsf{MA}_{\om_1}$ there are neither Suslin trees nor destructible ($\om_1$, $\om_1$)-gaps. We will see that we can add towers to this picture.


\section{Basic definitions}\label{basic}

We consider towers, i.e.\ families $(T_\alpha)_{\alpha<\kappa}$ such that $T_\alpha \setminus T_\beta$ is finite if and only if $\alpha\leq\beta$.
We do not assume that towers are maximal, $\kappa$ is always a regular cardinal, and we consider mainly towers of length $\om_1$. 
We say that two towers are equivalent if they generate (together with $\fin$) the same ideal in $\pom$. 

We shall define three properties of towers similar to properties used for classification of gaps. 
It is convenient to reveal some connections between towers and gaps first. 

Every gap consists of two towers and every tower is a half of a gap (the other half can be built by induction). Under $\mathsf{MA}_{\om_1}$ even more is true: every $\om_1$-tower is a half of $(\om_1,\om_1)$-gap (see \cite{spasojevic1} and
\cite[Remark 2.4]{shelah885}).
However, this is not a $\mathsf{ZFC}$ theorem. Indeed, if an $\om_1$-tower is maximal, then it could be only half of $(\om_1,1)$-gap. 

There are also $\om_1$-towers of different nature which cannot be half of an $(\om_1,\om_1)$-gap. If there is an $\omega_1$-scale (i.e.\ strictly $\leq^*$-increasing sequence $(f_\alpha)_{\alpha<\omega_1}$ of elements of $\om^\om$ eventually dominating all elements of $\om^\om$), then the tower defined by $T_\alpha =
\left\{(n,m)\colon m\leq f_\alpha(n)\right\}$ is not maximal (and its orthogonal is not generated by a single set), but it cannot be half of an $(\om_1,\om_1)$-gap. To see this, notice that every set in the orthogonal of $(T_\alpha)_{\alpha<\om_1}$ is a subset of $n\times \om$ for some $n\in\om$. Assume that $(T_\alpha,
R_\alpha)_{\alpha<\om_1}$ forms an $(\om_1,\om_1)$-pre-gap. Since there are only countably many choices of $n$, without loss of generality there is a fixed $n$ for which $R_\alpha \sub n\times \om$. Clearly, $n\times \om$ interpolates $(T_\alpha, R_\alpha)_{\alpha<\om_1}$.

We say that a tower of length $\kappa$ satisfies \emph{condition (K)} if $T_\alpha \nsubseteq T_\beta$ for each $\alpha, \beta<\kappa$.

\begin{defi}
	A tower $(T_\alpha)_{\alpha<\kappa}$ is special if it contains a cofinal subtower satisfying condition (K).
	A tower which is not special is called \emph{Suslin}.
\end{defi}

The name ``Suslin'' is justified by the fact that the poset $(\mathcal{T},\sub)$ contains neither uncountable $\sub$-chains nor uncountable $\sub$-antichains 
if $\mathcal{T}$ is a Suslin $\om_1$-tower. We will later see that if we add a tower
by a forcing, checking that this forcing is ccc is often the same as checking that the generic tower is Suslin.

We say that a tower $(T_\alpha)_{\alpha<\om_1}$ satisfies \emph{condition (H)} if
\[ \left\{\xi<\alpha\colon T_\xi \setminus T_\alpha \sub n\right\} \mbox{ is finite } \]
for each $\alpha<\om_1$ and $n< \om$.
(Note that this condition can not be directly generalized for~longer towers.)

\begin{defi}
	A tower $(T_\alpha)_{\alpha< \om_1}$ is Hausdorff if it contains a subtower satisfying condition (H).
\end{defi}

The following fact implies that Hausdorff towers are quite common. 

\begin{prop} \label{half-is-Hausdorff}
Let $(L_\alpha, R_\alpha)_{\alpha < \om_1}$ be a Hausdorff gap. The tower $(L_\alpha)_{\alpha < \om_1}$ is Hausdorff.
\end{prop}
\begin{proof}
Since $L_\alpha \cap R_\alpha = \emptyset$ for each $\alpha<\om_1$, for every $\alpha<\beta<\om_1$ if $L_\alpha\cap R_\beta \nsubseteq n$, then $L_\alpha \setminus L_\beta \nsubseteq n$.	
\end{proof}

Similarly one can prove the following:

\begin{prop} \label{half-is-special}
Let $(L_\alpha, R_\alpha)_{\alpha < \om_1}$ be a left-oriented gap. The tower $(L_\alpha)_{\alpha < \om_1}$ is special.
\end{prop}

The proof of the next fact is essentially the same as the proof of Fact~\ref{Hausdorff is left oriented}.

\begin{prop} \label{Hausdorff->noninclusion}
	If $(T_\alpha)_{\alpha<\om_1}$ satisfies condition (H), then there is an unbounded $X\sub \om_1$ such that $T_\alpha\setminus T_\beta\ne \emptyset$ for each distinct $\alpha$, $\beta\in X$.
\end{prop}
\begin{proof}
	Define $f\colon \om_1 \to [\om_1]^{<\om}$ by \[ f(\alpha) = \left\{\xi<\alpha\colon T_\xi \sub T_\alpha\right\}. \]
Hajnal's free set theorem implies that there is an unbounded $X\sub \om_1$ such that $\xi\notin f(\alpha)$ for each $\xi$, $\alpha\in X$. This means that $T_\xi \setminus T_\alpha\ne \emptyset$ for each $\xi < \alpha\in X$.
\end{proof}

\begin{cor}
	Every Hausdorff tower is special.
\end{cor}

In particular, Hausdorff gaps provide examples of uncountable towers which form anti-chains if ordered by $\sub$. Since Hausdorff gaps exist in $\mathsf{ZFC}$, it follows that special towers exist in $\mathsf{ZFC}$. 

There are facts indicating that the notion of a Hausdorff tower is more natural than the notion of a special tower. The next proposition shows that this is a ``global'' property, whereas Example \ref{special-nonHaus} will demonstrate that this is not the case of special
towers. (Another fact supporting the statement above is discussed in Section \ref{Tukey}.)

\begin{prop} \label{Hausdorff-equivalent}
	If a tower $(T_\alpha)_{\alpha < \om_1}$ is equivalent to a Hausdorff tower $(S_\alpha)_{\alpha< \om_1}$, then $(T_\alpha)_{\alpha< \om_1}$ is Hausdorff.
\end{prop}
\begin{proof}
	We can suppose that $(S_\alpha)_{\alpha<\om_1}$ satisfies condition (H).
	There exist some $n < \om$ and cofinal subtowers $(T'_\alpha)_{\alpha<\om_1}$ and $(S'_\alpha)_{\alpha<\om_1}$ such that 
	$S'_\alpha \setminus n \sub T'_\alpha \sub^* S'_{\alpha+1}$ for each $\alpha < \om_1.$
	Suppose that $(T'_\alpha)_{\alpha<\om_1}$ does not satisfy (H). There is some $\beta < \om_1$ and $m < \om$ such that 
	$I = \{ \xi < \beta \colon T'_\xi \setminus T'_\beta\sub m\}$ is infinite. 
	Fix $k > \max(n,m)$ such that $T'_\beta \sub S'_{\beta+1} \cup k.$ 
	Now $I \sub \{\xi < \beta+1 \colon S'_\xi \setminus S'_{\beta+1}\sub k\}$ and this contradicts (H) of $(S'_\alpha)_{\alpha<\om_1}.$
\end{proof}

The property of being special tower is invariant under a slightly stronger equivalence relation:

\begin{prop} \label{special-similar}
Assume $\lambda$ is a cardinal with uncountable cofinality and $\mathcal{T} = \left(T_\alpha\right)_{\alpha< \lambda}$ is a special tower.
If $\mathcal{T}' = \left(T'_\alpha\right)_{\alpha<\lambda}$ is such that $T_\alpha =^* T'_\alpha$ for each $\alpha<\lambda$, then $\mathcal{T}'$ is special.
\end{prop}
\begin{proof}
	There is $X\sub \lambda$ cofinal in $\lambda$ such that $(T_\alpha)_{\alpha\in X}$ is a $\sub$-antichain. We can find $X'\sub X$ cofinal in $\lambda$ such that 
both	$T'_\alpha \setminus T_\alpha$ and $T_\alpha \setminus T'_\alpha$ are constant for every $\alpha\in X'$. Clearly,
	$T'_\alpha \nsubseteq T'_\beta$ for every $\alpha<\beta$, $\alpha$, $\beta\in X'$.
\end{proof}

We finish this section by a comment on Proposition \ref{half-is-Hausdorff}.

\begin{remark} \label{maximal-Hausdorff}
We are not aware of any way of constructing (in $\mathsf{ZFC}$) a Hausdorff tower without producing a Hausdorff gap 
(i.e.\ without implicitly constructing the other half of the gap). However, there are several generic
examples of Hausdorff towers which are not halves of Hausdorff gaps. 
E.g.\ in every model obtained by forcing with a Suslin tree, there is a Hausdorff tower which is maximal. 
Let $\mathbb{S}$ be a Suslin tree. Define $\varphi\colon \mathbb{S} \to \mathcal{P}(\om)$ in such a way that  
\begin{itemize}
	\item[(1)] $\varphi(s)\cap \varphi(t) =^* \emptyset$ for each incompatible $s$, $t\in \mathbb{S}$;
	\item[(2)] $\varphi(s) \sub^* \varphi(t)$ if $t\leq s$;
	\item[(3)] if $S$ is a branch of the tree, then $\left\{\varphi(s)^c\colon s\in S\right\}$ satisfies condition (H).
\end{itemize}
Such $\varphi$ can be constructed by induction on levels of $\mathbb{S}$, using the fact that all branches of $\mathbb{S}$ are countable and a simplifying assumption that $\mathbb{S}$ does not split at limit levels. 
Having such $\varphi$, we can see that the $\mathbb{S}$-generic branch through $\varphi''[\mathbb{S}]$ is a tower which is maximal (in principle because $\mathbb{S}$ does not add new subsets of $\om$, see e.g.\ \cite[Lemma 2]{farah-towers}) and Hausdorff. 

This provides another example of a family generating a dense ideal 
which does not realize oscillation $1$ (cf. \cite[Example 1]{oscillations} and Section~\ref{suslin-towers} of this paper).

\end{remark}

\section{Special towers}\label{special-towers}

We already know that special towers do exist in $\mathsf{ZFC}$. We will see that consistently there are no other towers of length $\om_1$.
The simplest way to see this is to use $\mathsf{OCA}$. For the formulation of $\mathsf{OCA}$ see e.g.\ \cite{todorcevic}; we will only use the following consequence of $\mathsf{OCA}$ (see \cite[Proposition 8.4]{todorcevic}):

\begin{prop}[Todor{\v{c}}evi{\'c}]
Under $\mathsf{OCA}$ every uncountable subset of $\mathcal{P}(\om)$ contains an uncountable $\sub$-chain or $\sub$-antichain.
\end{prop}

A tower is well-ordered by $\sub^*$, so it cannot contain an uncountable chain. Hence the following holds:

\begin{prop} \label{OCA}
Under $\mathsf{OCA}$ every $\om_1$-tower is special.
\end{prop}

It is unclear for us whether $\mathsf{OCA}$ implies that all towers of length $\om_1$ are Hausdorff. However,
this is true if we assume $\mathsf{MA}_{\om_1}$. 

\begin{lem}\label{single_pairs}
 Let $(A_\alpha, B_\alpha)_{\alpha < \om_1}$ be a sequence such that $A_\alpha \sub B_\alpha \sub^* A_\beta \sub \om$ 
 and $A_{\alpha+1} \setminus B_\alpha$ is infinite for each for $\alpha < \om_1$.
 There exist $\xi < \zeta < \om_1$ such that $A_\xi \nsubseteq B_\zeta$.
\end{lem}
\begin{proof}
Suppose $A_\alpha \subseteq B_\beta$ for each $\alpha \leq \beta < \om_1$.
Put $C_\alpha = \bigcup \left\{A_\xi \colon \xi \leq \alpha\right\}$. 
We have $C_\alpha \sub B_\alpha$, and since $A_{\alpha+1} \setminus B_\alpha$ is infinite, $C_\alpha \neq C_{\alpha+1}$ 
for each $\alpha < \om_1.$ Thus $\left(C_\alpha\right)_{\alpha < \om_1}$ is an increasing $\subseteq$-chain of type $\om_1,$ a contradiction.  
\end{proof}

\begin{prop}\label{cccHausdorff}
	Let $\mathcal{T} = (T_\alpha)_{\alpha<\om_1}$ be a tower.
	There exists a ccc forcing making $\mathcal T$ Hausdorff in the extension.
\end{prop}
\begin{proof}
	A condition in the desired forcing is a pair $p = (F_p, n_p) \in [\om_1]^{<\om} \times \om.$ A condition $q$ is stronger than $p$ if $F_p \sub F_q$, $n_p \leq n_q$, and for each $\alpha < \beta$, $\alpha \in F_q\setminus F_p,$ $\beta \in F_p$ there exists some $m \in (T_\alpha \setminus T_\beta) \setminus n_p.$ 
	For each  condition $p$ and each ordinal $\alpha < \om_1,$ $F_p < \alpha,$ 
	the pair $\langle F_p \cup \{\alpha\}, n \rangle$ (where $n>n_p$) is a condition stronger than $p$, and thus this forcing adds a subtower cofinal in $\mathcal T$ which fulfills condition (H) (provided that $\om_1$ is preserved).
	
	To prove ccc, let $\left\{p_\alpha = \langle F_\alpha, n_\alpha \rangle \colon \alpha < \om_1\right\}$ be a set of conditions. 
	We can suppose that $n_\alpha = n$ for each $\alpha < \om_1$ and that $\{F_\alpha \colon \alpha < \om_1\}$ forms a
	$\Delta$-system with core $F$. Denote $F'_\alpha = F_\alpha \setminus F$. Assume without loss of generality that $\max F < \min F'_\alpha < \max F'_\alpha < \min F'_\beta$ if $\alpha<\beta$.
	
	For $\alpha < \om_1$ put $A_\alpha = \bigcap \left\{T_\xi \colon \xi \in F'_\alpha \right\} \setminus n$ 
	and $B_\alpha = \bigcup \left\{T_\xi \colon \xi \in F'_\alpha \right\} \setminus n$.
	Lemma~\ref{single_pairs} shows that there are $\alpha < \beta < \om_1$ such that $A_\alpha \not\sub B_\beta.$ Note that $\langle F_\alpha \cup F_\beta, n\rangle$ is a condition stronger than both $p_\alpha$ and $p_\beta.$
\end{proof}

\begin{cor}\label{MA-Hausdorff}
	($\mathsf{MA_{\om_1}}$) Every tower of length $\om_1$ is Hausdorff.
\end{cor}

Since there are no Hausdorff towers of length greater than $\om_1$, this result does not generalize to higher cardinals. However, the following is still true.

\begin{thm}\label{ccc-special}
	Let $\kappa$ be a regular uncountable cardinal and let $\mathcal T = (T_\alpha)_{\alpha< \kappa}$ be a tower.
	There is a ccc forcing making $\mathcal T$ special in the extension.
\end{thm}

Theorem \ref{ccc-special} can be proved in a similar way to \cite[Theorem 1.4]{KunenDouwen}, see also \cite[Theorem 4.4]{Weiss}. 
The forcing consists simply of finite subsets $F$ of $\kappa \setminus \gamma$ (for a suitably chosen $\gamma \in \kappa$) such that $T_\alpha \nsubseteq T_\beta$ if $\alpha\ne \beta$, and $\alpha$, $\beta\in F$. This forcing is ccc (checking it needs some work but it is not difficult). Therefore there is $\gamma < \kappa$ such that for any condition $F\subseteq \kappa \setminus \gamma$ there are cofinally many $\beta$ such that $F\cup \{\beta\}$ is a condition (otherwise we could construct an uncountable set of pairwise incompatible conditions). Hence this forcing adds a cofinal subtower satisfying~(K). Instead of proving Theorem~\ref{ccc-special} directly, we show a slightly stronger theorem. Namely, under $\mathsf{MA}_{\kappa}$ every tower of length $\lambda \leq \kappa$ (with $\lambda$ of uncountable cofinality) can be modified to a tower with condition (K) by a
minor cosmetic operation: it is enough to add at most one integer to each level and to remove at most one integer from each level. Proving ccc for this forcing is similar to proving it for the forcing mentioned above.

\begin{thm} \label{MA-similar}
    Let $\kappa$ be a regular uncountable cardinal and let $\mathcal{T} = \left(T_\alpha\right)_{\alpha<\kappa}$ be a tower. 
    There is a ccc forcing $\mathbb{P}$ which generically adds a tower $\mathcal{T}' = \left(T'_\alpha\right)_{\alpha<\kappa}$ 
    such that $|T_\alpha \setminus T'_\alpha|\leq 1$ and $|T'_\alpha\setminus T_\alpha|\leq 1$ for each $\alpha<\om_1$, and $\mathbb{P}$ forces that $\mathcal{T}'$ satisfies condition (K).
\end{thm}

This theorem together with Proposition \ref{special-similar} implies Theorem \ref{ccc-special}. 

\begin{lem}\label{multi_towers}
	For $k < \om$ and each $i < k$ let $\mathcal{T}_i = \left(T_\alpha^i\right)_{\alpha < \om_1}$ be a tower.
There exist $\zeta < \xi < \om_1$ such that $T_\zeta^i \nsubseteq T_\xi^i$ for each $i < k$.
\end{lem}


\begin{proof}
We prove the lemma by induction on $k$. The statement holds true for $k=1$.
At the $(k+1)$-th step use the induction hypothesis to find pairs $\zeta_\alpha < \xi_\alpha$ for $\alpha < \om_1$ 
such that $T_{\zeta_\alpha}^i \nsubseteq T_{\xi_\alpha}^i$ and $\xi_\alpha < \zeta_\beta$ for each $i < k$ and $\alpha < \beta < \om_1$.

\begin{claim}
We can moreover assume that $T_{\zeta_\alpha}^i \not\sub T_{\xi_\beta}^i$ for each $\alpha, \beta < \om_1$ and $i < k$. 
\end{claim}
We can first refine the system so that there is $n<\om$ such that
$T_{\zeta_\alpha}^i \cap n \nsubseteq T_{\xi_\alpha}^i \cap n$ for each $i < k$ and $\alpha < \om_1.$
After that refine further to get $T_{\zeta_\alpha}^i \cap n$ and $T_{\xi_\alpha}^i \cap n$ constant for a fixed $i$.
\qed\medskip

We are done if { $T_{\zeta_\alpha}^k \not\sub T_{\xi_\alpha}^k$} for some $\alpha < \om_1$, so suppose the opposite.
Lemma~\ref{single_pairs} states that there are $\alpha < \beta < \om_1$ such that $T_{\xi_\alpha}^k \not\sub T_{\zeta_\beta}^k$.
Thus $\xi = \xi_\alpha$ and $\zeta = \zeta_\beta$ are as required. 
\end{proof}

\begin{proof}[Proof of Theorem \ref{MA-similar}]


A condition $p \in \mathbb{P}$ is of the form $\left(F_p, a_p, r_p\right)$, where
\begin{itemize}
	\item $F_p\in [\kappa]^{<\om}$;
	\item $a_p\colon F \to \om$ and $r_p\colon F \to \om$;
	\item for every $\alpha < \beta\in F$ we have
	\[T_\alpha \cup \left\{a_p(\alpha)\right\} \setminus \left\{r_p(\alpha)\right\} \nsubseteq T_\beta \cup \left\{a_p(\beta)\right\} \setminus \left\{r_p(\beta)\right\}.\]
\end{itemize}
The ordering is given by $q\leq p$ if $F_p \sub F_q$, $a_q|F_p = a_p$ and $r_q|F_p = r_p$. Notice that for each condition $p\in \mathbb{P}$ and $\alpha< \kappa$ there is $q\in \mathbb{P}$ such that $\alpha\in F_q$ and $q\leq p$. Indeed, choose
\[ m\notin \bigcup\left\{T_\xi\cup \left\{a_p(\xi)\right\} \colon \xi\in F_p\setminus \alpha\right\}\] 
and 
\[ n\in \bigcap\left\{T_\xi\setminus \left\{r_p(\xi)\right\} \colon \xi\in F_p\cap \alpha\right\} \setminus \{m\} \]
and define $F_q = F_p \cup \{\alpha\}$, $a_q(\alpha) = m$, $r_q(\alpha) = n$. Let $G$ be a $\mathbb{P}$-generic, $a = \bigcup_{p\in G} a_p$, $r = \bigcup_{p\in G} r_p$. Clearly, the tower defined by \[T'_\alpha = T_\alpha \cup \{a(\alpha)\} \setminus \{r(\alpha)\}\] is as desired.


It only remains to show that our forcing is ccc.
Let $\{p_\alpha \colon \alpha < \om_1\}$ be a set of conditions. We will denote $F_\alpha = F_{p_\alpha}$, $a_\alpha = a_{p_\alpha}$ and $r_\alpha = r_{p_\alpha}$. 
By thinning out the sequence if necessary, we may assume that
 $F_\alpha = \left\{\xi_\alpha^0 < \xi_\alpha^1 < \ldots < \xi_\alpha^{k-1}\right\}$
 for each $\alpha<\om_1$ and that $a_\alpha\left(\xi^i_\alpha\right)$ and~$r_\alpha\left(\xi^i_\alpha\right)$ depend only on $i$.
Using the $\Delta$-lemma we further assume that $F_\alpha = F \cup F'_\alpha$ for each $\alpha<\om_1$, where
$(F'_\alpha)_{\alpha<\om_1}$ is pairwise disjoint and there is $I\sub k$ such that $F = \left\{\xi^i_\alpha \colon i\in I\right\}$ for every $\alpha<\om_1$. 
So for each $i<k$ the sequence $(\xi^i_\alpha)_{\alpha<\om_1}$ is either constant or injective.  
Considering a subsequence once again (if necessary), we may assume that $(\xi^i_\alpha)_{\alpha < \om_1}$ is either constant or strictly increasing for each $i<k$.	
We may also assume that there is $l<\omega$ such that the sequence 
\[ \left( \left( (T_{\xi^i_\alpha} \cup \{a_\alpha(\xi^i_\alpha)\} ) \setminus \{ r_\alpha(\xi^i_\alpha)\}\right) \cap l\right)_{\alpha<\omega_1} \]
is constant for each $i < k$, where $l$ is such that 
\[ \left( \left(T_{\xi^i_\alpha} \cup \{a_\alpha(\xi^i_\alpha)\}\right) \setminus \{r_\alpha(\xi^i_\alpha)\}\right) \cap l\]
 and  
\[ \left( \left(T_{\xi^j_\alpha}  \cup \{a_\alpha(\xi^j_\alpha)  \}\right) \setminus \{r_\alpha(\xi^j_\alpha  )\}\right) \cap l \]
are $\subseteq$-incompatible for $i \neq j$.
Apply Lemma \ref{multi_towers} to find $\alpha$, $\beta$ such that 
\[T_{\xi^i_\alpha} \cup \{ a_\alpha(\xi^i_\alpha)\}) \setminus \{ r_\alpha(\xi^i_\alpha) \} \nsubseteq T_{\xi^i_\beta} \cup \{ a_\beta(\xi^i_\beta) \} \setminus \{ r_\beta(\xi^i_\beta) \}\] for each $i\in k\setminus I$.
Now $q = \left(F_\alpha \cup F_\beta, a_\alpha \cup a_\beta, r_\alpha \cup r_\beta\right)$ is a condition in $\mathbb{P}$, and $q\leq p_\alpha$, $q\leq p_\beta$. 
\end{proof}

Proposition \ref{cccHausdorff} is an analogue of the theorem stating that under $\mathsf{MA}_{\om_1}$ every gap is Hausdorff. In \cite{abraham} the authors prove that the same statement holds assuming the P-ideal dichotomy. This is not true for towers. The P-ideal
dichotomy is compatible with $\mathsf{CH}$ and under $\mathsf{CH}$ Suslin towers do exist. However, if we additionally assume that $\mathfrak{b}$ is big, the P-ideal dichotomy implies that every $\om_1$-tower is Hausdorff. Recall that $\mathfrak{b}$ is the minimal cardinality of a family in $\om^\om$ which cannot be $\leq^*$-dominated by a single function. The \emph{P-ideal dichotomy} ($\mathsf{PID}$) is the assertion: for every P-ideal $\mathcal{I} \sub [\om_1]^\om$ one of the
following holds:
\begin{itemize}
	\item there is an uncountable $K\sub \om_1$ such that $[K]^\om \sub \mathcal{I}$;
	\item $\om_1 = \bigcup_{n<\om} A_n$ and $A_n\cap I$ is finite for each $n< \om$ and  $I\in \mathcal{I}$.
\end{itemize}
Notice that if 
for each uncountable $K\sub \om_1$ there is an infinite $I\sub K$, $I\in \mathcal{I}$, then the~second alternative cannot hold.

\begin{thm}\label{PID}
Assume $\mathsf{PID}$. Every $\om_1$-tower is Hausdorff if and only if $\mathfrak{b}>\om_1$.
\end{thm}

\begin{remark*}
A related result with a similar proof was obtained independently in~\cite{Stevo-Dilip}.
Namely:
\begin{thm*}
Assume $\mathsf{PID}$. The following are equivalent.
\begin{enumerate}
	\item $\min \left\{\mathfrak b, \cof(\mathcal F_\sigma) \right\} > \om_1$.
	\item There are only 5 Tukey types of directed sets of size at most $\om_1$.
\end{enumerate}
\end{thm*}

For the definition of the cardinal invariant $\cof(\mathcal F_\sigma)$ see~\cite{Stevo-Dilip}.
The relation of these results becomes apparent in Section~\ref{Tukey}, where it is shown that Hausdorff towers correspond to the Tukey type $\left[\om_1\right]^{<\om}$.

\end{remark*}

\begin{proof}
In the next section we shall prove that a Suslin tower of length $\mathfrak{b}$ always exists (Proposition \ref{b-towers}). We prove here only the ``if'' part of the~theorem.

Define an ideal $\mathcal{I} \sub [\om_1]^{\leq \om}$ by 
\[ I \in \mathcal{I} \mbox{ iff } C^n_\alpha(I) = \left\{\xi\in \alpha\cap I\colon T_\xi \setminus T_\alpha \sub n\right\} \mbox{ is finite for each }\alpha<\om_1, n< \om. \]
\begin{claim} If $\mathfrak{b}>\om_1$, then $\mathcal{I}$ is a P-ideal.\end{claim}

Consider a sequence $\{I_n \colon n< \om\} \in [\mathcal{I}]^\om$. Assume without loss of generality that $(I_n)_{n<\om}$ 
is pairwise disjoint, and fix an enumeration $I_n = \left\{\xi^n_k\colon k< \om\right\}$ for each $n$. For~every $\alpha<\om_1$ define a function $f_\alpha \colon \om\to \om$
by\[ f_\alpha(n) = \max \left\{k \colon T_{\xi^n_k} \setminus T_\alpha \sub n\right\}. \]
Let $g\colon \om\to\om$ be a function $\leq^*$-dominating $\{f_\alpha\colon \alpha<\om_1\}$. Let 
\[ I = \bigcup_{n< \om} I_n \setminus \left\{\xi^n_k\colon k\leq g(n)\right\}. \]
It is straightforward to check that $I\in \mathcal{I}$ and $I_n \sub^* I$ for each $n$. \qed
\medskip

The first alternative of $\mathsf{PID}$ for $\mathcal{I}$ gives us a subtower which fulfills condition (H),
so we only need to refute the second alternative of $\mathsf{PID}$. We shall show that for each uncountable $K\sub \om_1$ there is $I \in \mathcal I \cap {[K]}^\om.$



\begin{claim}
	There exists $x \in 2^\om$ such that $x \in \overline{\{ T_\alpha \colon \alpha \in K\}}$ (the closure in the Cantor space)
	but~$x \notin \langle T_\alpha \rangle_{\alpha < \om_1}$ (the ideal generated by the tower). 
\end{claim}
If $\overline{\{ T_\alpha \colon \alpha \in K\}} \sub \langle T_\alpha \rangle_{\alpha < \om_1}$, 
then the ideal $\langle T_\alpha \rangle_{\alpha < \om_1}$ is generated by a closed set and thus it is an analytic P-ideal.
On the other hand, an analytic P-ideal which is not countably generated cannot be generated by less than $\mathfrak d$-many sets \cite[Theorem 6]{Stevo-anal-quotients}. 
\qed\medskip

Fix $I \in [K]^\om$ such that $x$ is the single accumulation point of $\{T_\alpha \colon \alpha \in I\}$.
To conclude that $I \in \mathcal I$, notice that if for some $\beta < \om_1$ and $n < \om$ we have $T_\alpha \sub T_\beta \cup n$ for infinitely many $\alpha \in I$,  there would be an accumulation point of $\{T_\alpha \colon \alpha \in I\}$ which would be a subset of $T_\beta \cup n$ and hence in $\langle T_\alpha \rangle_{\alpha < \om_1}$.
\end{proof}

This seems a convenient moment at which to mention the following two results. Note that none of them directly implies Corollary \ref{MA-Hausdorff}.

\begin{thm} [Shelah \cite{shelah885}]
$\mathsf{MA_{\om_1}}$ implies that every $\om_1$-tower is a right half of~a~Hausdorff gap. 
\end{thm}

\begin{thm}[Spasojevi\'c \cite{spasojevic1}]\label{thm-spasojevic}
	$\mathsf{MA}_{\omega_1}(\sigma\text{-centered})$ implies that every $\om_1$-tower is a~right half of~a~left-oriented gap. 
\end{thm}

We present ideas behind the proof of the above theorem in Section \ref{gaps-towers} (see Example \ref{David}). 


\section{Suslin towers}\label{suslin-towers}


We know that consistently there are no Suslin $\om_1$-towers. However, Suslin towers, perhaps longer than $\om_1$, always exist:

\begin{prop}\label{b-towers}
	There is a tower $\mathcal{T} = \left(T_\alpha\right)_{\alpha < \mathfrak{b}}$ which is Suslin.
\end{prop}
\begin{proof}
Let $\left\{f_\alpha\colon \alpha<\mathfrak{b}\right\} \sub \om^\om$ be a $\leq^*$-unbounded family which is $\leq^*$-strictly increasing. Define
\[ T_\alpha = \left\{(m,n)\colon n\leq f_\alpha(m)\right\} \]
for every $\alpha$. This is a $\mathfrak{b}$-tower (on $\om\times\om$). If $K\sub \mathfrak{b}$ is cofinal, then $\{f_\alpha\colon \alpha\in K\}$ is $\leq^*$-unbounded, and thus there is $\alpha<\beta$, $\alpha$, $\beta\in K$ such that $f_\alpha(m)\leq f_\beta(m)$ for each $m< \om$ (see \cite{oscillations-functions}). 
Therefore $T_\alpha \sub T_\beta$ and $(T_\alpha)_{\alpha<\mathfrak b}$ is a Suslin tower.
\end{proof}

The above fact and Theorem \ref{PID} may suggest that the existence of a Suslin $\om_1$-tower is equivalent to $\mathfrak{b}=\om_1$ in $\mathsf{ZFC}$. This is not the case.

\begin{prop}\label{s-centered-Suslin}
Let $\kappa$ be an uncountable regular cardinal.
It is consistent that $\mathfrak{b}=\kappa$ and there is a Suslin $\om_1$-tower.
\end{prop}
\begin{proof}
	Start with a model of $\mathfrak{b}=\om_1$ with a Suslin tower. 
	Then use a finite support iteration of Hechler forcings $\mathbb{H}$ (for adding a dominating real) of length $\kappa$. This will make $\mathfrak{b}=\kappa$ in the extension.  Hechler forcing is $\sigma$-centered and thus it possesses the Knaster property (i.e.\ for every
	uncountable $X\subseteq \mathbb{H}$ there is an
	uncountable linked $X_0 \subseteq X$), which is preserved in finite-support iterations. 

	We will prove that a forcing with the Knaster property does not destroy a Suslin tower $(T_\alpha)_{\alpha<\omega_1}$.
	(This also follows from the general well known fact that such forcing preserves ccc-ness of ground model relations.)
	Suppose that $\mathbb{P}$ is such a forcing, $p \in \mathbb P$ is any condition, and $\dot{X}$ is a $\mathbb{P}$-name for an uncountable subset of
	$\om_1$. 
	Consider 
	\[ X = \left\{\alpha<\om_1\colon \exists p_\alpha < p, p_\alpha \Vdash \alpha\in \dot{X}\right\}. \]
	There is 
	an uncountable $X_0\sub X$ such that $p_\alpha \parallel p_\beta$ for each $\alpha, \beta \in X_0$.
	The tower $(T_\alpha)_{\alpha<\om_1}$ is Suslin, hence there are distinct $\alpha, \beta \in X_0$ such that $T_\alpha \sub T_\beta$ and any $q<p_\alpha, p_\beta$
	forces that $\alpha, \beta \in \dot{X}$.
	Therefore the tower remains Suslin in the extension.
\end{proof}

The crux of Proposition \ref{b-towers} is Todor\v{c}evi{\'c}'s result on oscillations of functions. His work on oscillations of subsets of $\om$ in \cite{oscillations} sheds even more light on the existence of Suslin towers. Recall that an oscillation of
$A$, $B\sub \om$ (denoted by $\mathrm{osc}(A,B)$) is the cardinality of the set $A\bigtriangleup B/\sim$, where $\sim$ is the equivalence relation defined on $A\bigtriangleup B$ by 
\[ m \sim n \mbox{ iff } [n,m] \cap (A \bigtriangleup B) \sub A\setminus B \mbox{ or } [n,m] \cap (A \bigtriangleup B) \sub B\setminus A \]
 (We slightly abuse the notation treating $[n,m]$ as $[m,n]$ for $m<n$.) We say that a family $\mathcal{A}\sub \mathcal{P}(\om)$ \emph{realizes an oscillation $n$} if there are $A,B\in \mathcal{A}$ such that $\mathrm{osc}(A,B)=n$.

The following is a special case of~\cite[Corollary 2]{oscillations}.

\begin{thm}[Todor{\v{c}}evi{\'c}, \cite{oscillations}] \label{osci}
	If a family $\mathcal{A}\sub \mathcal{P}(\om)$ generates a non-meager P-ideal, then it realizes all finite oscillations.
\end{thm}

Notice that if $A\subsetneq^* B$, then $A\sub B$ if and only if $\mathrm{osc}(A,B) = 1$. It follows that each tower generating a non-meager ideal is Suslin. 
We enclose here for the reader's convenience the sketch of the proof of the latter assertion (extracted from \cite{oscillations}):

\begin{proof} We will say that a tower $\mathcal{T} = (T_\alpha)_{\alpha<\kappa}$ has \emph{property ($\xi$)} 
	if for an arbitrarily large $n< \om$ there is $t\sub n$ such that for each $m>n$ there are arbitrarily large $\beta<\kappa$ with the properties
\begin{itemize}
	\item $T_\beta \cap n = t$;
	\item $[n, m)\sub T_\beta$.
\end{itemize}

\begin{claim} ($\xi$) Let $\mathcal{T}$ be a tower of size $\kappa$ of uncountable cofinality such that every cofinal subtower of $\mathcal T$ has property ($\xi$).
This $\mathcal{T}$ is a Suslin tower.
\end{claim}

This is basically \cite[Lemma 2]{oscillations}. 
Since $P(\om)$ is hereditary separable, we can fix a countable set $\mathcal D\sub \mathcal{T}$ dense in $\mathcal{T}$.
There is $\alpha<\kappa$ such that $D\sub^* T_\alpha$ for each $D\in \mathcal D$. Without loss of generality we
can assume that there is $m_0< \om$ such that $T_\alpha\setminus m_0 \sub T_\beta$ for every $\beta>\alpha$.
Using property ($\xi$) we can fix $m_1>m_0$ and $t\sub m_1$ such that for every $m>m_1$ there is $\beta>\alpha$
such that $T_\beta \cap m_1 = t$ and $[m_1,m) \sub T_\beta$.

Pick $D\in [t]\cap \mathcal D$. Fix $m>m_1$ such that $D\setminus m \sub T_\alpha$ and $\beta$ such that $[m_1,m) \sub T_\beta$.

Then 
\begin{itemize}
	\item $D\cap m_1 = t = T_\beta \cap m_1$;
	\item $D\cap [m_1,m) \sub [m_1,m) = T_\beta \cap [m_1,m)$;
	\item $D\setminus m \sub T_\alpha \setminus m_0 \sub T_\beta$.
\end{itemize}
Hence $D\sub T_\beta$.
\qed
\bigskip

It is enough to show that every tower which generates a non--meager ideal has property ($\xi$).
This is basically the beginning of the proof of \cite[Theorem 1]{oscillations} and the proof of \cite[Lemma 1]{oscillations}.
We may assume that for each finite $F\sub \omega$ the set $\{\alpha\colon F \sub T_\alpha\}$ is either empty or cofinal in $\mathcal{T}$. This is standard (since $[\omega]^{<\om}$ is countable). Then we argue \emph{a contrario}. Subsequently negating
($\xi)$ we obtain an increasing sequence of natural numbers $(n_k)_{k<\om}$ witnessing the fact that $\mathcal{T}$ generates a meager ideal.
\end{proof}

As a corollary we obtain many examples of Suslin towers. E.g.\ each tower generating a maximal ideal is Suslin.

In a somewhat similar manner (to Claim ($\xi$)) we can prove that adding a Cohen real adds a Suslin tower. This result is not a surprise, the proof mimics the well known argument used by Todor\v{c}evi\'{c} to show that Cohen reals produce destructible gaps. 

\begin{prop}\label{cohen-inclusion-tower}
Let $\left(T_\alpha\right)_{\alpha < \om_1}$ be a tower and let $c$ be a Cohen generic real in an extension.
Then $\left(T_\alpha \cap c\right)_{\alpha < \om_1}$ is a Suslin tower.
\end{prop}
\begin{proof}
To see that the tower is not eventually constant notice that $c \cap \left(T_\beta \setminus T_\alpha\right)$ is infinite for each $\alpha<\beta < \om_1.$

Let $p \in \lxp{n}{2}$ be a Cohen condition and $\dot{X}$ be a name for an uncountable subset of~$\om_1.$ We can assume that $X = \dot{X}$ belongs to the ground model (by taking a stronger condition if necessary). 
Consider $\alpha<\beta$ ($\alpha$, $\beta \in X$) such that $T_\alpha \cap n = T_\beta \cap n$ and fix $m>n$ such that $T_\alpha \subset T_\beta \cup m.$
Extend $p$ to $q \in \lxp{m}{2}$ such that $q^{-1}(1) = p^{-1}(1).$ 
Now $q \Vdash T_\alpha \cap \dot{c} \sub T_\beta \cap \dot{c}.$
\end{proof}

This simple example is of some importance, since the resulting Suslin tower will be used in the next section to produce a special non-Hausdorff gap.
Notice also that intersecting a Cohen real with a gap gives us a destructible gap with both sides being Suslin towers. So it is possible to have Suslin towers which are far from being non-meager (whose orthogonal is not generated by a single set).

One way to add a tower generically is to use a standard technique inspired by Hechler's work \cite{Hechler}. It allows one to prove (see e.g.\ \cite[Theorem 5.8, Chapter 2]{frankiewicz}) that whenever $\mathcal{P}$ is a partial order, there is a forcing notion $\mathbb{P}$ such that $\mathbb{P} \Vdash `` \check{\mathcal{P}} \text{
embedds in } \pofin "$. It seems that whenever $\mathcal{P}$ is a partial order and $\mathcal{C}\sub \mathcal{P}$ is an uncountable chain, then
in Hechler's extension the embedding of $\mathcal{C}$ into $\pofin$ will be Suslin unless we impose some additional restrictions on the conditions of $\mathbb{P}$. We will try to justify this by examples below and in the next section.

\begin{example} The classical Hechler's forcing for adding a tower. \label{hechler-tower}

	A condition in $\mathbb{P}$ is a triple $p=\left(F_p,n_p,A_p\right)$ where $F_p \in [\om_1]^{<\om},$
	$n_p < \om$ and $A_p \sub F_p \times n_p.$
	For two conditions $p$, $q$ we use notation $p\cup q = \left(F_p\cup F_q, n_p \cup n_q, A_p \cup A_q\right).$
	A condition $q$ is stronger than $p$ if $n_p \leq n_q,$ $F_p\sub F_q,$
	$A_q \cap \left(F_p \times n_p\right) = A_p$
	and for each $\alpha,\beta \in F_p,$ $\alpha < \beta $ and $i\in \left[n_p, n_q\right)$ 
	\[\text{ if } (\alpha, i) \in A_q \text{, then } (\beta,i) \in A_q.\]


\begin{claim}
	$\mathbb{P}$ is ccc.
\end{claim}
	Fix a set of conditions $\{p_\alpha \colon \alpha < \om_1\}.$ Use the $\Delta$-lemma to find an uncountable set $I$ such that $\{F_{p_\alpha} \colon \alpha \in I\}$ forms a $\Delta$-system with core $\Delta$ and $n_{p_\alpha}$ is constant for $\alpha
	\in I.$ We can further refine $I$ to an uncountable $I'$ so that $A_{p_\alpha} \cap (\Delta \times n_{p_\alpha})$ is constant.
	Now for each $\alpha,\beta \in I'$ the conditions $p_\alpha$ and $p_\beta$ are compatible since $p_\alpha \cup p_\beta$ is their common extension.\qed
	
\medskip
	Let $G$ be a generic filter. Put $A = \bigcup_{p \in G} A_p.$
	For $\alpha < \om_1$ define \[T_\alpha = \{i < \om \colon (\alpha, i)\in A \}.\]
\begin{claim}
	$\left(T_\alpha\right)_{\alpha<\om_1}$ is a Suslin tower.
\end{claim}
	It is obvious that $\left(T_\alpha\right)_{\alpha<\om_1}$ is non-constant.
	Consider a name $\dot{X}$ for an uncountable subset of $\om_1$ and a condition $p.$ There is an uncountable set 
	\[X = \left\{\alpha < \om_1 \colon \exists p_\alpha < p, \alpha \in F_{p_\alpha}, p_\alpha \Vdash \alpha \in \dot{X}\right\}.\]
	Now proceed in the same way as in the proof of the previous claim to get an uncountable set $I \subseteq X.$
	We may further suppose that $\left\{i<n_{p_\alpha}\colon (\alpha, i) \in A_{p_\alpha}\right\}$ 
	is constant for $\alpha \in I.$
	Hence~$p_\alpha \cup p_\beta < p_\alpha, p_\beta$ and 
	\[ p_\alpha \cup p_\beta \Vdash \left(\alpha,\beta \in \dot{X} \text{ and } 
	T_\alpha \sub T_\beta\right) \] for $\alpha, \beta \in I, \alpha<\beta.$ 
	\qed
\end{example}

The forcing in this example is in fact equivalent to the forcing adding $\om_1$ Cohen reals. In what follows we denote the latter by $\mathbb{C}_{\om_1}$.

\begin{prop} \label{equivalent-to-Cohen}
	$\mathbb{P}$ is equivalent to $\mathbb{C}_{\om_1}$.
\end{prop}
\begin{proof}
Using \cite[Main Theorem]{Koppelberg}, it is enough to find a sequence of 
$(\mathbb{P}_\alpha)_{\alpha<\omega_1}$ such that
\begin{enumerate}
	\item $\mathbb{P}_\gamma = \bigcup_{\alpha<\gamma} \mathbb{P}_\alpha$ for each limit $\gamma\leq \omega_1$,
	\item\label{koppel-embeddings} for $\alpha<\beta$, $\mathbb{P}_\alpha$ is a complete suborder of $\mathbb{P}_\beta$,
	\item $\mathbb{P}_{\alpha+1}/\mathbb{P}_\alpha$ is equivalent to Cohen forcing.
\end{enumerate}

For $\alpha<\om_1$ let $\mathbb{P}_\alpha = \{(F,n,A)\in \mathbb{P}\colon F\sub \alpha\}$. 
Only checking (\ref{koppel-embeddings}) is non-trivial.

It is enough to show that for $\alpha < \beta \leq \om_1$ there is a pseudo-projection $p\colon\mathbb{P}_\beta \to \mathbb{P}_\alpha$ (see \cite[Proposition 2]{Balcar-Pazak}). 
I.e., we need to define for each $q = \left(F_q, n_q, A_q\right) \in \mathbb{P}_\beta$ a condition $p(q) \in \mathbb{P}_\alpha$ 
such that whenever $r<p(q), r \in \mathbb{P}_\alpha$ then $r$ is compatible with $q$ (in $\mathbb{P}$).
It is trivial to check that $p(q) = \left(F_q \cap \alpha, n_q, A_q \cap (\alpha \times n_q)\right)$ works.
\end{proof}

In what follows we will present several other incarnations of $\mathbb{C}_{\om_1}$ used for producing peculiar towers and gaps. 

\begin{example}\label{gen-haus} Hechler's forcing with the Hausdorff restriction.

	Consider a modification of the forcing from Example \ref{hechler-tower}.
	We add one additional requirement for $q<p.$
	Namely, for each $\alpha \in F_p$ and $\xi \in F_q \setminus F_p$, $\xi < \alpha$,
	there has to be some $i \geq  n_p$ such that $(\xi, i) \in A_q$ and $(\alpha, i) \notin A_q.$
	
	This forcing adds a generic tower $(T_\alpha)_{\alpha<\omega_1}$ satisfying condition (H) in the same way as forcing from Example~\ref{hechler-tower} adds a Suslin tower. As in Example \ref{hechler-tower} we can show that this forcing is equivalent to $\mathbb{C}_{\om_1}$ (and so is ccc),
	the same definition of $\mathbb{P}_\alpha$ pseudo-projections works also for this forcing.
	Notice however, that checking this is not as trivial as before (but not difficult either).

	Notice also that the tower added by this forcing is maximal (and so this is another example of a maximal Hausdorff tower, see
	Remark \ref{maximal-Hausdorff}). Indeed, let $P\subseteq \om$ be an infinite set from the
	extension. It is enough to check that $P$ intersects some $T_\alpha$ on an infinite set. 
	Because of ccc, the name for $P$ is guessed on some intermediate step 
	so we can forget about an initial segment of the tower, and assume that $P$ is from the ground model.
	Then the set \[ D_n = \left\{p\in \mathbb{P} \colon 0 \in F_p \ \exists m>n \ (0,m)\in A_p \mbox{ and } m\in P\right\} \]
	is dense in $\mathbb{P}$ for each $n$. This proves that $P \cap T_0$ is infinite in the extension.

\end{example}
\medskip

Probably the most interesting example of this sort is the following one.

\begin{example} \label{special-nonHaus} A special tower equivalent to a Suslin tower.
	
Let $\kappa$ be an uncountable regular cardinal. 
 We construct a forcing which adds a pair of equivalent towers of length $\kappa$,
 one of them being special and the other one Suslin (in a strong sense).
 
 A condition is a sequence $p=\left(F_p,n_p, \left\langle T_p^\alpha, S_p^\alpha \right\rangle_{\alpha \in F_p} \right),$
	where $F_p \in [\kappa]^{<\om},$
	$n_p < \om$, and $T^\alpha_p, S^\alpha_p \sub n_p$ for each $\alpha \in F_p$, and 
	 $S^\alpha_p \not\sub S^\beta_p$ for $\alpha < \beta \in F_p.$
 
 A condition $q$ is stronger than $p$ if $n_p \leq n_q,$ $F_p \sub F_q,$
	$T^\alpha_q \cap n_p = T^\alpha_p$, $S^\alpha_q \cap n_p = S^\alpha_p$ for $\alpha \in F_p$,
	and for each $\alpha,\beta \in F_p,$ $\alpha < \beta $ and $i \in [n_p, n_q)$ 
	\[\text{ if } i \in T^\alpha_q \cup S^\alpha_q \text{ then } i \in T^\beta_q \cap S^\beta_q \text{\quad and \quad}
	\text{ if } i \in T^\alpha_q \text{ then } i \in S^\alpha_q.\]
	
	It is easy to see that for each $\alpha < \kappa$ the set $\{p \colon \alpha \in F_p\}$ 
	is dense and hence this forcing adds a couple of equivalent towers of length $\kappa^V$
	$(T_\alpha)_{\alpha < \kappa}$ and $(S_\alpha)_{\alpha < \kappa}$ defined by $T_\alpha = \bigcup_{p \in G} T^\alpha_p$ and
	$S_\alpha = \bigcup_{p \in G} S^\alpha_p$ for $\alpha < \kappa.$
	
	The tower $(S_\alpha)_{\alpha < \kappa}$ satisfies condition (K). On the other hand $(T_\alpha)_{\alpha<\kappa}$ is far from being special.
	
\begin{claim}
 Every uncountable subtower of $(T_\alpha)_{\alpha < \kappa}$ is Suslin.
\end{claim}

Consider a name $\dot{X}$ for an uncountable subset of $\kappa$ and a condition $p.$ There is an~uncountable set 
\[X = \left\{\alpha < \kappa \colon \exists p_\alpha < p,\ \alpha \in F_{p_\alpha},\ p_\alpha \Vdash \alpha \in \dot{X} \right\}.\]
	
	Use the $\Delta$-lemma to find an uncountable set $I$ such that $\{F_{p_\alpha} \colon \alpha \in I\}$ 
	forms a `nice' $\Delta$-system with core $\Delta$.
	Each $F_{p_\alpha}, \alpha \in I$ is split into blocks 
	\[ F_{p_\alpha} = F^0_\alpha \cup \Delta^0 \cup F^1_\alpha \cup \Delta^1 \cup \ldots \cup F^{k-1}_\alpha \cup \Delta^{k-1},\]
	$\Delta = \bigcup \Delta^i,$ $\max F_\alpha^{i} < \min \Delta^i,$ $\max \Delta^{i-1} < \min F_\alpha^{i},$
	$\max F_\alpha^i < \min F_\beta^i$, and 
	\[ F_\alpha^i = \left\{\xi_0^i(\alpha) < \xi_1^i(\alpha) < \ldots <  \xi_{j(i)-1}^i(\alpha) \right\} \]
	for each $\alpha < \beta < \om_1$ ($\alpha$, $\beta \in I$) and $i < k.$
	($F^0_\alpha$ and some $\Delta^{i}$s may be empty, in that case disregard the required inequalities.)
	
	We may moreover assume that $T^\xi_{p_\alpha}$ and $S^\xi_{p_\alpha}$ are constant for any $\xi\in \Delta$, 
	that $n_{p_\alpha}$, $T_{p_\alpha}^{\xi^i_m(\alpha)}$ and $S_{p_\alpha}^{\xi^i_m(\alpha)}$ 
	are constant (ranging over $\alpha \in I$)
	for each $i<k, m<j(i)$, and that there are $J, M < \om$ such that 
	$\alpha = \xi^J_M(\alpha)$ for~$\alpha \in I$.

	Pick any $\alpha < \beta \in I$. Define condition $q$ by $F_q = F_{p_\alpha} \cup F_{p_\beta},$
	$n_q = n_{p_\alpha}+k+1$ and define
	\begin{itemize}

	\item for $i < J$ and $\chi \in F^i_\alpha \cup \Delta^i$ let 
		\[	T^\chi_q = T^\chi_{p_\alpha} \cup [n_{p_\alpha}, n_{p_\alpha} +i +1)\mbox{ and }S^\chi_q = S^\chi_{p_\alpha} \cup [n_{p_\alpha}, n_{p_\alpha} +i +1).\]
	
	\item for $i < J$ and $\chi \in F^i_\beta$ let 
		\[			T^\chi_q = T^\chi_{p_\beta} \cup [n_{p_\beta}, n_{p_\beta} +i)\mbox{ and }S^\chi_q = S^\chi_{p_\beta} \cup [n_{p_\beta}, n_{p_\beta} +i).\]
	
	\item for $i > J$ and $\chi \in F^i_\alpha \cup \Delta^i$ let 
		\[	T^\chi_q = T^\chi_{p_\alpha} \cup [n_{p_\alpha}, n_{p_\alpha} +i +2)\mbox{ and }S^\chi_q = S^\chi_{p_\alpha} \cup [n_{p_\alpha}, n_{p_\alpha} +i +2).\]

	\item for $i > J$ and $\chi \in F^i_\beta$ let 
		\[ T^\chi_q = T^\chi_{p_\beta} \cup [n_{p_\beta}, n_{p_\beta} +i +1)\mbox{ and }S^\chi_q = S^\chi_{p_\beta} \cup [n_{p_\beta}, n_{p_\beta} +i +1).\]

	\item for $\chi \in \Delta^J$ let 
		\[ T^\chi_q = T^\chi_{p_\alpha} \cup [n_{p_\alpha}, n_{p_\alpha} +J +2)\mbox{ and }S^\chi_q = S^\chi_{p_\alpha} \cup [n_{p_\alpha}, n_{p_\alpha} +J +2).\]

	\item for $m < M$ and $\chi = \xi^J_m(\alpha)$ let 
		\[ T^\chi_q = T^\chi_{p_\alpha} \cup [n_{p_\alpha}, n_{p_\alpha} +J +1)\mbox{ and }S^\chi_q = S^\chi_{p_\alpha} \cup [n_{p_\alpha}, n_{p_\alpha} +J +1).\]

	\item for $m < M$ and $\chi = \xi^J_m(\beta)$ let 
		\[ T^\chi_q = T^\chi_{p_\beta} \cup [n_{p_\beta}, n_{p_\beta} +J)\mbox{ and }S^\chi_q = S^\chi_{p_\beta} \cup [n_{p_\beta}, n_{p_\beta} +J).\]
	
	\item for $m > M$ and $\chi = \xi^J_m(\alpha)$ let 
		\[ T^\chi_q = T^\chi_{p_\alpha} \cup [n_{p_\alpha}, n_{p_\alpha} +J +2)\mbox{ and }S^\chi_q = S^\chi_{p_\alpha} \cup [n_{p_\alpha}, n_{p_\alpha} +J +2).\]

	\item for $m > M$ and $\chi = \xi^J_m(\beta)$ let 
		\[ T^\chi_q = T^\chi_{p_\beta} \cup [n_{p_\beta}, n_{p_\beta} +J +1)\mbox{ and }S^\chi_q = S^\chi_{p_\beta} \cup [n_{p_\beta}, n_{p_\beta} +J +1).\]

	\item for $\chi = \xi^J_M(\alpha)$ let
		\[ T^\chi_q = T^\chi_{p_\alpha} \cup [n_{p_\alpha}, n_{p_\alpha} +J +1)\mbox{ and }S^\chi_q = S^\chi_{p_\alpha} \cup [n_{p_\alpha}, n_{p_\alpha} +J +2).\]
	
	\item for $\chi = \xi^J_M(\beta)$ let
		\[ 	T^\chi_q = T^\chi_{p_\beta} \cup [n_{p_\beta}, n_{p_\beta} +J +1)\mbox{ and }S^\chi_q = S^\chi_{p_\beta} \cup [n_{p_\beta}, n_{p_\beta} +J +1).\]
\end{itemize}

	To show that $q$ is a condition of $\mathbb{P}$ is straightforward. 
	Condition $q$ is a common extension of both $p_\alpha$ and $p_\beta$,
	$q \Vdash \alpha,\beta \in \dot{X}$. and 
	$q \Vdash \dot{T_{\alpha}} \subseteq \dot{T_{\beta}}.$
	\qed
	\medskip


	The proof of the claim also shows that the forcing is ccc.
	In case $\kappa = \om_1$ this forcing is equivalent to $\mathbb{C}_{\om_1}$. To check this, use the same strategy as in the proof of Proposition~\ref{equivalent-to-Cohen}.
	Define $\mathbb{P}_\alpha = \{q\in \mathbb{P}\colon F_q \sub \om\cdot \alpha\}$ for $\alpha < \om_1$. 
    For $\gamma < \beta$  and $q \in \mathbb{P}_\beta$ define a pseudo-projection $p(q) \in \mathbb{P}_\gamma$ as follows.
	First find a set $F \sub \om \cdot \gamma$ such that $|F| = |F_q \setminus (\om \cdot \gamma)|$ and $F_q \cap (\om \cdot \gamma) < F$,
	and fix an order preserving bijection $b\colon F_q \setminus (\om \cdot \gamma) \to F$.
	Define \[ p(q)=\left(F_{p(q)} = \left( F_q \cap (\om \cdot \gamma)\right) \cup F,n_q, \left\langle T_{p(q)}^\alpha, S_{p(q)}^\alpha \right\rangle_{\alpha \in F_{p(q)}} \right),\]
	where \[\left(T_{p(q)}^\alpha, S_{p(q)}^\alpha\right) = \left(T_{q}^\alpha, S_{q}^\alpha\right)\] for $\alpha \notin F$, 
	and \[ \left(T_{p(q)}^\alpha, S_{p(q)}^\alpha\right) = \left(T_{q}^{b^{-1}(\alpha)}, S_{q}^{b^{-1}(\alpha)}\right)\] for $\alpha \in F$. 
	
	We will sketch the proof that $p(q)$ is a pseudo-projection. Suppose that $r<p(q)$ and $r\in \mathbb{P}_\gamma$. We want to find $s\in \mathbb{P}_\beta$ such that $s<r$ and $s<q$. Let $F_s = F_r \cup F_q$, $n_s = n_r+1$. 
	For $\eta\in F_q \setminus (\om\cdot \gamma)$ set $T^\eta_s = T^{b(\eta)}_r$. 
	If $\xi\leq\max(F_q\cap (\om\cdot\gamma))$ or $\xi\geq  \om\cdot\gamma$ let $n_r\notin T^\xi_s \cup S^\xi_s$ and for $\xi\in (\max(F_q\cap (\om\cdot\gamma)),\om\cdot\gamma)$ let $n_r\in T^\xi_s\cap S^\xi_s$. 
	We have to show that for each $\xi\in F_r$ and $\eta\in F_q\setminus (\om\cdot \gamma)$ we have $S^\xi_s \nsubseteq S^\eta_s$. If $\xi\in (\max(F_q\cap (\om\cdot\gamma)), \om\cdot\gamma)$, then $n_r \in S^\xi_s \setminus S^\eta_s$. If $\xi\leq \max(F_q\cap (\om\cdot\gamma))$, then $S^\xi_s \nsubseteq S^\eta_s$
	since $S^\xi_{r} \nsubseteq S^{b^{-1}(\eta)}_{r}$. Hence $s\in \mathbb{P}_\beta$. It is easy to check that $s<r$ and $s<q$.

\end{example}
\medskip

This example refutes the natural conjecture that each special tower is in fact Hausdorff (since a Hausdorff tower cannot be equivalent to a Suslin tower). Moreover, it proves that the property of being special, unlike the Hausdorff property, is not
invariant under the equivalence of towers (cf.\ Proposition \ref{Hausdorff-equivalent}).

In the following section we provide another example of a tower of this kind: a tower which is neither Hausdorff nor equivalent to a Suslin tower.

Notice that most of the examples presented in this section exist in models obtained by adding $\om_1$ Cohen reals. It seems that the structure of towers is particularly rich in such models. We will show that adding $\om_1$ Cohen reals produces various interesting gaps.

\section{Structure of gaps after adding $\om_1$ Cohen reals} \label{gaps-towers}

One of the most natural questions related to destructibility of gaps is asking whether the class of special $(\om_1,\om_1)$-gaps coincides
with the class of Hausdorff gaps. It was posed in \cite{scheepers} as Problem 2. Since we isolated 
another property lying in between of the above ones, we can ask more specifically: 


\begin{prob}\cite[Problem 1]{scheepers} \label{special-oriented}
Is every special gap left-oriented?
\end{prob}

\begin{prob}\label{oriented-Hausdorff}
Is every left-oriented gap equivalent to a Hausdorff gap?
\end{prob}

Hirschorn in \cite{hirschorn} answered Scheeper's problem. More precisely, he gave an example of a left-oriented gap which is not equivalent to any Hausdorff gap, so he answered in negative Problem \ref{oriented-Hausdorff}. It turns out that the answer to
Problem \ref{special-oriented} is also negative.

\begin{thm} \label{special-nonHausdorff1}
There is a special gap which is not left-oriented.
\end{thm}

First, we give an example which relies only on simple facts and known results. In particular, we need the following theorem due to Roitman:

\begin{thm}[\cite{roitman}, \cite{roitman1}]\label{thm-roitman}
Adding a single Cohen real to a model satisfying $\mathsf{MA}(\sigma\text{-centered})$ preserves $\mathsf{MA}(\sigma\text{-centered})$.
\end{thm}

\begin{example}An inverted Spasojevi\'{c} gap.\label{special-nonLO}

Work in a model of $\mathsf{MA}_{\omega_1}(\sigma\text{-centered})$.
Using Proposition~\ref{cohen-inclusion-tower} and the theorem above, we can add a Cohen real and get a Suslin tower
$\left(R_\alpha\right)_{\alpha < \om_1}$ in the extension without loosing $\mathsf{MA}(\sigma\text{-centered})$. Of course, the tower cannot be maximal since $\mathfrak{t}>\om_1$. Theorem~\ref{thm-spasojevic} now gives us a special gap $\left(L_\alpha, R_\alpha\right)_{\alpha <
\om_1}$ fulfilling condition (O). Consider the gap $(R_\alpha, L_\alpha)_{\alpha <  \om_1}$.
Inverting the sides of an indestructible gap cannot make it destructible, so $\left(R_\alpha, L_\alpha\right)_{\alpha<\om_1}$ is still special.
However, it cannot be left-oriented. Indeed, in this case Proposition~\ref{half-is-special} would imply that $\left(R_\alpha\right)_{\alpha<\om_1}$ is special,  but this tower is Suslin.
\end{example}

The reader perhaps wonder if the gap $(L_\alpha, R_\alpha)_{\alpha<\om_1}$ introduced by the forcing from Theorem~\ref{thm-spasojevic} is Hausdorff.
We will show that it is not. Actually, Example~\ref{David} will show that gaps introduced by Spasojevi\'c's forcing are left-oriented, but not
Hausdorff. Thus to obtain a special non-Hausdorff gap, we do not need to invert the gap in Example~\ref{special-nonLO}. As a corollary we obtain that left-oriented gaps are not necessarily right-oriented. The following example shows that the
Hausdorff condition for gaps is not symmetric either. There is a Hausdorff gap such that the inversed gap is not Hausdorff. We start Hechler's machinery again.


\begin{example} An asymmetric Hausdorff gap. \label{uninvertible-haus}

 
We define a forcing $\mathbb{P}$ consisting of conditions of the form
\[ p=\left(F_p,n_p, \left\langle L_p^\alpha, R_p^\alpha \right\rangle_{\alpha \in F_p} \right),\]
where 
\begin{enumerate}[(1)]
	\item $F_p \in [\om_1]^{<\om}$; \label{uninvertible-haus-cond_first}
	\item $n_p < \om$;
	\item $L^\alpha_p, R^\alpha_p \sub n_p$ for each $\alpha \in F_p$;
	\item $L^\alpha_p \cap R^\alpha_p = \emptyset$ for each $\alpha \in F_p$. \label{uninvertible-haus-cond_last}
\end{enumerate}
	
	A condition $q$ is stronger than $p$ if 
	
	\begin{enumerate}[(a)]
		\item $n_p \leq n_q$ and $F_p\sub F_q$; \label{uninvertible-haus-ord_a}
		\item $L^\alpha_q \cap n_p = L^\alpha_p$, $R^\alpha_q \cap n_p = R^\alpha_p$ for $\alpha \in F_p$;
		\item for each $\alpha,\beta \in F_p,$ $\alpha < \beta $ and $i\in [n_p, n_q)$ 
	\[\text{ if } i \in L^\alpha_q \text{ then } i \in L^\beta_q \text{\quad and \quad}
	\text{ if } i \in R^\alpha_q \text{ then } i \in R^\beta_q;\] \label{uninvertible-haus-ord_c}
		\item for each $\alpha \in F_p$ and $\xi \in F_q \setminus F_p,$ $\xi < \alpha$
	there is some $i \geq  n_p$ such that $i \in L^\xi_q \cap R^\alpha_q$.
	\end{enumerate}
	It is easy to see that for each $\alpha < \om_1$ the set $\{p \colon \alpha \in F_p\}$ 
	is dense. Let $G$ be a $\mathbb{P}$-generic filter, and let $L_\alpha = \bigcup_{p \in G} L^\alpha_p$ and
	$R_\alpha = \bigcup_{p \in G} R^\alpha_p$ for $\alpha < \om_1$. Then $(L_\alpha, R_\alpha)_{\alpha<\om_1}$ is Hausdorff 
	provided $\mathbb{P}$ preserves $\om_1$.
	
\begin{claim}
	$\mathbb{P}$ is equivalent to adding $\om_1$ Cohen reals. 
\end{claim}

As in Proposition \ref{equivalent-to-Cohen}, the $\mathbb{P}_\beta$ consists of conditions 
$q = \left(F_q,n_q, \left\langle L_q^\alpha, R_q^\alpha \right\rangle_{\alpha \in F_q} \right)$ with $F_q \sub \beta$. 
The pseudo-projection $p \colon \mathbb{P}_\beta \to \mathbb{P}_\gamma$ is defined by
\[p(q) = \left(F_q \cap \gamma,n_q, \left\langle L_q^\alpha, R_q^\alpha \right\rangle_{\alpha \in F_q \cap \gamma} \right).\]

\qed

\begin{claim}
$(R_\alpha)_{\alpha < \om_1}$ is a Suslin tower.
\end{claim}

Consider a name $\dot{X}$ for an uncountable subset of $\om_1$ and a condition $p\in \mathbb{P}$. There is an uncountable set 
\[X = \left\{\alpha < \om_1 \colon \exists p_\alpha < p,\ \alpha \in F_{p_\alpha},\ p_\alpha \Vdash \alpha \in \dot{X}\right\}.\]
	We will proceed in the same way as in the examples from the previous section.
	Use the $\Delta$-lemma to find an uncountable set $I$ such that $\{F_{p_\alpha} \colon \alpha \in I\}$ 
	forms a $\Delta$-system with core $\Delta,$ $\max \Delta < \min F_{p_\alpha} \setminus \Delta$ for
	$\alpha \in I$,	and $n_{p_\alpha} = n_\Delta$ is constant for $\alpha \in I$. We may assume that $\max F_{p_\alpha} < \min F_{p_\beta} \setminus \Delta$ for $\alpha<\beta<\om_1$. 
	We can further refine $I$ to an uncountable $I'$ so that all $R^\alpha_{p_\alpha}$, 
	$L^\xi_{p_\alpha}$ and $R^\xi_{p_\alpha}$ are constant for all $\xi \in \Delta$, $\alpha \in I'$.
	Pick any $\alpha < \beta \in I' \setminus \Delta$, and define a condition $q$ by $F_q = F_{p_\alpha} \cup F_{p_\beta},$
	$n_q = n_{\Delta}+1,$ 
	\begin{enumerate}[(i)]
		\item $L^\xi_q = L^\xi_{p_\alpha}$ and $R^\xi_q = R^\xi_{p_\alpha}$ for $\xi \in \Delta$, \label{uninvertible-haus-ext_i}
		\item $L^\xi_q = L^\xi_{p_\alpha} \cup \{n_\Delta\}$ and $R^\xi_q = R^\xi_{p_\alpha}$ for $\xi \in F_{p_\alpha} \setminus \Delta$,
		\item $L^\xi_q = L^\xi_{p_\beta}$ and $R^\xi_q = R^\xi_{p_\beta} \cup \{n_\Delta\}$ for $\xi \in F_{p_\beta} \setminus \Delta$. \label{uninvertible-haus-ext_iii}
	\end{enumerate}
	
	Condition $q$ is a common extension of both $p_\alpha$ and $p_\beta$, $q \Vdash \alpha, \beta \in \dot{X}$ and 
	$q \Vdash \dot{R_\alpha} \subseteq \dot{R_\beta}.$

	To show that $\mathbb{P}$ is ccc, we do the same reductions for an arbitrary uncountable set of conditions. 
	\qed
\end{example}

We present now another example witnessing the negative answer for Problem~\ref{special-oriented}.

\begin{example}A special gap which is neither left- nor right-oriented.\label{ex-non-oriented}
 	
We define a forcing $\mathbb{P}$ similar to the poset from the previous example (and also equivalent to $\mathbb{C}_{\om_1}$). 
	A condition $p\in \mathbb{P}$ is of the form $p=\left(F_p,n_p, \left(L_p^\alpha, R_p^\alpha\right)_{\alpha \in F_p} \right)$ and
	it satisfy the properties (\ref{uninvertible-haus-cond_first}--\ref{uninvertible-haus-cond_last}) 
	from Example~\ref{uninvertible-haus}. We impose the following additional restriction:
	\begin{itemize}
		\item  $\left(L^\alpha_p \cap R^\beta_p\right) \cup \left(L^\beta_p \cap R^\alpha_p\right) \neq \emptyset$ for each $\alpha < \beta \in F_p.$
	\end{itemize}
	The ordering on $\mathbb{P}$ is defined by conditions (\ref{uninvertible-haus-ord_a}--\ref{uninvertible-haus-ord_c}) 
	from the previous example.

	Let $G$ be a $\mathbb{P}$-generic filter. Put $L_\alpha = \bigcup_{p \in G} L^\alpha_p$ and $R_\alpha = \bigcup_{p \in G} R^\alpha_p$ for $\alpha < \om_1.$
    It is clear that $\left(L_\alpha,R_\alpha\right)_{\alpha<\om_1}$ is a special gap. 


\begin{claim}
 Both $\left(L_\alpha\right)_{\alpha< \om_1}$ and $\left(R_\alpha\right)_{\alpha< \om_1}$ are Suslin towers.
\end{claim}
	We prove it for the right side, the proof for the left side is exactly the same.
	Consider a name $\dot{X}$ for an uncountable subset of $\om_1$ and a condition $p$. There is an uncountable set 
	\[X = \left\{\alpha < \om_1 \colon \exists p_\alpha < p,\ \alpha \in F_{p_\alpha},\ p_\alpha \Vdash \alpha \in \dot{X}\right\}.\]
	Now proceed in the same way as in Example \ref{uninvertible-haus} to get an uncountable set $I \subseteq X$.
	Pick $\alpha < \beta \in I \setminus \Delta$ and define a condition $q$ by 
	$F_q = F_{p_\alpha} \cup F_{p_\beta}$,
	$n_q = n_{p_\alpha}+1$, and by (\ref{uninvertible-haus-ext_i}--\ref{uninvertible-haus-ext_iii}) 
	from the previous example. Condition $q$ is a common extension of both $p_\alpha$ and $p_\beta$,
	$q \Vdash \alpha, \beta \in \dot{X}$ and $q \Vdash \dot{R_\alpha} \subseteq \dot{R_\beta}.$
	\qed
    \medskip

The proof that this forcing is equivalent to $\mathbb{C}_{\om_1}$ works in a similar way as in Example~\ref{special-nonHaus}.
Let $\mathbb{P}_\beta$ be generated by conditions $q$ such that $F_q \sub \om\cdot \beta$, and define the pseudo-projection in the same way as in Example~\ref{special-nonHaus}.

	\end{example}

We prove now that consistently there is a gap providing answers to both questions from the beginning of this section. 


\begin{thm} \label{David0}
In a model obtained by adding $\om_1$ Cohen reals there is a gap $\left(L_\alpha, R_\alpha\right)_{\alpha<\om_1}$ such that
\begin{itemize}
	\item $\left(L_\alpha, R_\alpha\right)_{\alpha<\om_1}$ is left-oriented but not Hausdorff; 
	\item $\left(R_\alpha, L_\alpha\right)_{\alpha<\om_1}$ is special but not left-oriented. 
\end{itemize}
\end{thm}

\begin{proof}
Define a forcing notion equivalent to adding $\om_1$ Cohen reals which forces the existence of the desired gap.
A condition in $\mathbb{P}$ is a sequence \[p=\left(F_p, n_p, \langle L_p^\alpha, R_p^\alpha \rangle_{\alpha \in F_p} \right)\]
satisfying properties (\ref{uninvertible-haus-cond_first}--\ref{uninvertible-haus-cond_last}) of Example \ref{uninvertible-haus} and such that additionally
\begin{itemize}
	\item $L^\alpha_p \cap R^\beta_p \neq \emptyset$ for each $\alpha < \beta \in F_p$.
\end{itemize}
The ordering of $\mathbb{P}$ is defined by (\ref{uninvertible-haus-ord_a}--\ref{uninvertible-haus-ord_c}) of Example \ref{uninvertible-haus}.
	
As in the previous examples, it is easy to see that $\mathbb{P}$ adds a generic gap which is left-oriented (provided $\om_1$ is preserved).

\begin{claim} $\mathbb{P}$ is equivalent to $\mathbb{C}_{\om_1}$ (and so it is ccc).
\end{claim}

This is exactly the same proof as in Example~\ref{ex-non-oriented} (which is in turn similar to the proof from Example~\ref{special-nonHaus}).
\qed

	
\begin{claim}
	The tower $(L_\alpha)_{\alpha < \om_1}$ is not Hausdorff and the tower $(R_\alpha)_{\alpha<\om_1}$ is Suslin. 
\end{claim}

We prove both statements simultaneously. 
We need to show that there is no cofinal subtower $(L_\alpha)_{\alpha \in \dot{X}}$ satisfying the condition (H).
Consider a name $\dot{X}$ for an uncountable subset of $\om_1$, and suppose that
some condition $p$ forces that $(L_\alpha)_{\alpha \in \dot{X}}$ satisfies (H). We show that this leads to~a~contradiction, and $(L_\alpha)_{\alpha<\om_1}$ is not Hausdorff. Moreover, we will prove that there is $q<p$ and $\alpha<\beta\in \dot{X}$ 
such that $q \Vdash \dot{R}_\alpha \sub \dot{R}_\beta$, showing that $(R_\alpha)_{\alpha<\om_1}$ is Suslin.

There is an uncountable set 
\[I = \left\{\alpha < \om_1 \colon \exists p_\alpha < p, \alpha \in F_{p_\alpha}, p_\alpha \Vdash \alpha \in \dot{X}\right\}.\]
	
Using the $\Delta$-lemma we may assume that $ \Delta < F'_\alpha = F_{p_\alpha} \setminus \Delta$, 
and $n_{p_\alpha} = n$ is constant for $\alpha \in I$. Moreover, $F'_\alpha < F'_\beta$ for $\alpha < \beta$, and
$L^\xi_{p_\alpha}$ and $R^\xi_{p_\alpha}$ are constant for all $\xi \in \Delta$.
For some $\ell< \om$ we have $F'_\alpha = \{\xi^0_\alpha < \xi^1_\alpha < \ldots < \xi^{\ell-1}_\alpha\}$
for each $\alpha$, and $L^{\xi^i_\alpha}_{p_\alpha} = L^{\xi^i_\beta}_{p_\beta}$, 
$R^{\xi^i_\alpha}_{p_\alpha} = R^{\xi^i_\beta}_{p_\beta}$ for $\alpha, \beta \in I$, $i \in \ell$.
Finally, there is $i' < \ell$ for which $\alpha = \xi^{i'}_{p_\alpha}$ for all $\alpha \in I$.

Let $\alpha_0$ be the first element of $I$ and let $\beta \in I$ be some ordinal with infinitely many 
predecessors in $I$. Define condition $q$ by $F_q = F_{\alpha_0} \cup F_{\beta}$, $n_q = n+1$ and
\begin{itemize}
	\item $L_q^\xi = L_{p_\beta}^\xi$ for $\xi \in F_{p_{\beta}}$,
 \item $L_q^\xi = L_{p_{\alpha_0}}^\xi \cup \{n\}$ for $\xi \in F'_{\alpha_0}$,
 \item $R_q^\xi = R_{p_{\alpha_0}}^\xi$ for $\xi \in F_{p_{\alpha_0}}$,
 \item $R_q^\xi = R_{p_\beta}^\xi \cup \{n\}$ for $\xi \in F'_\beta$.
\end{itemize}

It is straightforward to check that $q\in \mathbb{P}$ and $q< p_{\alpha_0}, p_\beta$. Notice also that $q \Vdash \dot{R}_{\alpha_0} \sub \dot{R}_{\beta}$ (at this point we already know that $(R_\alpha)_{\alpha<\om_1}$ is Suslin).
According to our assumption on $\dot{X}$, there exist some $k < \om$ and a condition $r < q$ such that
\[r \Vdash \left|\left\{\alpha \in \dot{X} \cap \beta\colon L_\alpha \setminus L_\beta \sub n+1 \right\}\right|<k.\]
Since $F_r$ is finite, we can find $\{\alpha_1 < \alpha_2 < \ldots < \alpha_k \} \sub I\cap \beta$
such that $\alpha_0 < \alpha_1$ and
\[ F_r \cap [\min F'_{\alpha_1}, \max F'_{\alpha_k}] = \emptyset. \]
Define condition $s$ by $F_s = F_r \cup \bigcup_{j \leq k} F'_{\alpha_j}$, $n_s = n_r+k$ and
\begin{itemize}
 \item $L^\xi_s = L^\xi_r$ for $\xi \in F_r$, $\xi \leq \max \Delta$,
 \item $L^\xi_s = L^\xi_r \cup \{n_r\}$ for $\xi \in F_r$, $\max \Delta < \xi < \min F'_{\alpha_1}$, 
 \item $L^\xi_s = L^\xi_r \cup \big[n_r, n_r+k\big)$ for $\xi \in F_r$, $\max F'_{\alpha_k} < \xi$,
 \item $L^{\xi^i_{\alpha_j}}_s = (L^{\xi^i_{\alpha_0}}_r \cup \{n_r+j\})\cap n_s$ for $i<\ell$, $j \leq k$,
 \item $R^\xi_s=R^\xi_r$ for $\xi \in F_r$,
 \item $R^{\xi^i_{\alpha_j}}_s = R^{\xi^i_{\alpha_0}}_r \cup \big[n_r,n_r+j\big)$ for $i<\ell$, $j \leq k$.
\end{itemize}
It is not difficult to verify that $s\in \mathbb{P}$ and $s<r$, $s<p_{\alpha_j}$ for each $j \leq k$.
Hence 
\[s \Vdash \{\alpha_0,\alpha_1,\ldots, \alpha_k, \beta\} \sub \dot{X}.\]
Moreover $L^{\alpha_j}_s \setminus L^\beta_s = \{n\}$ for each $j \leq k$, and thus
$s\Vdash L_{\alpha_j} \setminus L_\beta = \{n\}$.
But this is a contradiction with $s<r$.\qed
\medskip

Notice that Proposition \ref{half-is-Hausdorff} implies that $(L_\alpha, R_\alpha)_{\alpha<\om_1}$ is not Hausdorff (but it is left-oriented). Moreover, the gap $(R_\alpha, L_\alpha)_{\alpha<\om_1}$ is still special, but Proposition \ref{half-is-special} implies that it cannot be left-oriented. 
\end{proof}

In fact, by a slight modification of the above proof, we can show that the original ($\sigma$-centered) forcing of Spasojevi\v{c} from \cite{spasojevic1} also produces a left-oriented non-Hausdorff gap.

\begin{example} A left-oriented gap not equivalent to any Hausdorff gap. \label{David}

Let $\mathcal R = \{R_\alpha \colon \alpha < \om_1\}$ be a given tower.
Spasojevi\v{c} introduced\footnote{In fact, Spasojevi\v{c} dealt with gaps in $\left(\fc, <^*\right)$ rather than $\left([\om]^\om, \subset^*\right)$ but the construction is analogous.} 
a $\sigma$-centered forcing $\mathbb{P}$ adding a tower $(L_\alpha)_{\alpha < \om_1}$ such that
$(L_\alpha, R_\alpha)_{\alpha < \om_1}$ is an oriented gap. 
We show that the tower $(L_\alpha)_{\alpha < \om_1}$ is not Hausdorff. 

A condition in $\mathbb{P}$ is a triple \[p=\left( F_p,n_p, \left(L_p^\alpha\right)_{\alpha \in F_p} \right),\]
\begin{enumerate}[(1)]
	\item $F_p \in [\om_1]^{<\om}$;
	\item $n_p < \om$;
	\item $L^\alpha_p \sub n_p$ for each $\alpha \in F_p$;
	\item $L^\alpha_p \cap R_\alpha = \emptyset$ and $L^\alpha_p \cap R_\beta \neq \emptyset$ for each $\alpha<\beta \in F_p$;
	\item $R_\alpha \setminus R_\beta \sub n_p$ for each $\alpha<\beta\in F_p$.
\end{enumerate}
	
A condition $q$ is stronger than $p$ if 
	
\begin{enumerate}[(a)]
		\item $n_p \leq n_q$ and $F_p\sub F_q$;
		\item $L^\alpha_q \cap n_p = L^\alpha_p$ for $\alpha \in F_p$;
		\item for each $\alpha < \beta \in F_p$ we have $L^\alpha_q \cap \left[n_p,n_q\right)\sub L^\beta_q$.
\end{enumerate}

\begin{lem}
 $\mathbb{P}$ is $\sigma$-centered.
\end{lem}

This is proved in \cite{spasojevic1} in more detail for an analogous forcing. We present a sketch of the argument for the reader's convenience.

\begin{proof}
For each $\gamma \leq \om_1$ define forcing $\mathbb P_\gamma$ consisting of conditions 
\[p=\left( F_p, G_p, n_p, \left(L_p^\alpha\right)_{\alpha \in G_p} \right),\]
\begin{enumerate}[(1)]
	\item $F_p \in [\om_1]^{<\om}$, $G_p \subseteq F_p \cap \gamma$;
	\item $n_p < \om$;
	\item $L^\alpha_p \sub n_p$ for each $\alpha \in G_p$;
	\item $L^\alpha_p \cap R_\alpha = \emptyset$ and $L^\alpha_p \cap R_\beta \neq \emptyset$ for each $\alpha<\beta$, $\alpha \in G_p$, $\beta \in F_p$.
\end{enumerate}
	
A condition $q$ is stronger than $p$ if 
	
\begin{enumerate}[(a)]
		\item $n_p \leq n_q$, $F_p\sub F_q$ and $G_p \sub G_q$;
		\item $L^\alpha_q \cap n_p = L^\alpha_p$ for $\alpha \in G_p$;
		\item for each $\alpha < \beta \in G_p$ we have $L^\alpha_q \cap \left[n_p,n_q\right)\sub L^\beta_q$;
		\item $L^\alpha_q \cap \left[n_p,n_q\right) \cap R_\beta = \emptyset$ for each $\alpha \in G_p$, $\beta \in F_p$.
\end{enumerate}

It is easy to check that $\mathbb P_{\om_1}$ is a forcing equivalent with $\mathbb P$.

\begin{claim}
	$\mathbb P_\gamma \subseteq \mathbb P_\delta$ is a regular embedding for $\gamma < \delta \leq \om_1$.
\end{claim}
\begin{proof}
	The inclusion is an embedding of posets.
	To show regularity	define $\pi \colon \mathbb P_\delta \to \mathbb P_\gamma$ by 
	\[\left( F, G, n, \left(L^\alpha\right)_{\alpha \in G} \right) \mapsto 
	\left( F, G\cap \gamma, n, \left(L^\alpha\right)_{\alpha \in G\cap \gamma} \right).\]
	It is straightforward to check that $\pi$ is a pseudo-projection from $\mathbb P_\delta $ to $\mathbb P_\gamma$.
\end{proof}

\begin{claim}
	$\mathbb P_\gamma$ is $\sigma$-centered for each $\gamma < \om_1$.
\end{claim}
\begin{proof}
	The set of conditions sharing the same $G$, $n$ and $\left(L^\alpha\right)_{\alpha \in G}$ is centered.
 \end{proof}

To conclude the proof notice that $\mathbb P_{\om_1}$ is a direct limit of the sequence of $\sigma$-centered posets
 $\left\{\mathbb P_\gamma \colon \gamma < \om_1 \right\}$ and hence is $\sigma$-centered.
\end{proof}

In the generic extension define $L_\alpha = \bigcup_{p \in G} L^\alpha_p$.
Now $(L_\alpha, R_\alpha)_{\alpha < \om_1}$ is an oriented gap.

\begin{claim}
 The tower $(L_\alpha)_{\alpha < \om_1}$ is not Hausdorff. Consequently, the gap $(L_\alpha, R_\alpha)_{\alpha < \om_1}$ is not equivalent to a Hausdorff gap.
\end{claim}
\begin{proof}
	For a contradiction, take a name $\dot{X}$ for an uncountable subset of $\om_1$ and assume that
	some condition $p$ forces that $(L_\alpha)_{\alpha \in \dot{X}}$ satisfies (H).

There is an uncountable set 
\[I = \left\{\alpha < \om_1 \colon \exists p_\alpha < p, \alpha \in F_{p_\alpha}, p_\alpha \Vdash \alpha \in \dot{X}\right\}.\]

Fix some large enough cardinal $\theta$ and countable elementary submodels $M, N \prec H(\theta)$, $I, \mathcal R \in N, M$ such that $M\in N$.
Notice that $M\sub N$ (since $M$ is countable), and $M\ne N$.
Fix some $\alpha_0 \in I \cap N \setminus M$. 

Work in $M$. By passing to a subset we can suppose that all conditions $p_\alpha$ for $\alpha \in I$ are isomorphic to $p_{\alpha_0}$ and form a
`nice' $\Delta$-system with core $\Delta = F_{\alpha_0} \cap M$.
In particular, we assume that $ \Delta < F'_\alpha = F_{p_\alpha} \setminus \Delta$ 
and $n_{p_\alpha} = n_{p_0} = n$ for~$\alpha \in I$. Moreover, $F'_\alpha < F'_\beta$ for~$\alpha < \beta$ and
$L^\xi_{p_\alpha} \cap n$, and~$R_\xi \cap n$ are constant for~all~$\xi \in \Delta$.
For some $\ell< \om$ we have $F'_\alpha = \{\xi^0_\alpha < \xi^1_\alpha < \ldots < \xi^{\ell-1}_\alpha\}$
for each $\alpha$, and $L^{\xi^i_\alpha}_{p_\alpha}\cap n = L^{\xi^i_\beta}_{p_\beta} \cap n$, 
$R_{\xi^i_\alpha}\cap n = R_{\xi^i_\beta}\cap n$ for $\alpha, \beta \in I$, $i < \ell$.
Finally, there is $i' < \ell$ for which $\alpha = \xi^{i'}_{p_\alpha}$ for all $\alpha \in I$.

For $\alpha \in I$ denote \[\overline{R}_\alpha = \bigcup \left\{R_\xi \colon \xi \in F'_\alpha \right\}
\text{\quad and \quad} \underline{R}_\alpha= \bigcap \left\{R_\xi \colon \xi \in F'_\alpha \right\}.\]

Fix $\beta \in I \setminus N$. 
There is some $n_0 > n$, $n_0 \in \underline{R}_\beta$ such that $n_0 \not\in \overline{R}_{\alpha_0}$.
Define a condition $q$ by $F_q = F_{\alpha_0} \cup F_\beta$, 
 $n_q = \max\left(n_0, \overline{R}_{\alpha_0} \setminus \underline{R}_\beta\right)+1$ and
\begin{itemize}
	\item $L_q^\xi = L_{p_\beta}^\xi$ for $\xi \in F_\beta$, 
 	\item $L_q^\xi = L_{p_{\alpha_0}}^\xi \cup \{n_0\}$ for $\xi \in F'_{\alpha_0}$.
\end{itemize}

Thus $q< p_{\alpha_0}, p_\beta$.
There exist some $k < \om$ and a condition $r < q$ such that
\[r \Vdash \left|\left\{\alpha \in \dot{X} \cap \beta\colon L_\alpha \setminus L_\beta \sub n_q \right\}\right|<k.\]

Denote $A = F_r \cap N$, $B = F_r\setminus N$,
and let $R_A = \bigcup \{R_\xi \colon \xi \in A\}$, $R_B = \bigcup \{R_\xi \colon \xi \in B\}$,
$R_r = R_A \cup R_B$.

\begin{claim}
There exist a sequence $\{\alpha_1 < \alpha_2 < \ldots < \alpha_k\} \sub I \cap N$, $\max A < \min F'_{\alpha_1}$ such that $R_{\xi^i_{\alpha_j}} \cap n_r = R_{\xi^i_{\alpha_0}} \cap n_r$ for $j \leq k$, $i < \ell$,
and a sequence $\{n_j \colon n_j > n_r, 0<j\leq k\} \sub \om \setminus R_r$ such that for
$0<j\leq k, i \leq k$ we have $n_j \in \underline{R}_{\alpha_i}$ if $j\leq i$ and 
$n_j \not\in \overline{R}_{\alpha_i}$ if $j > i$.
\end{claim}
\begin{proof}
	Let $I_0$ be such that 
	$R_{\xi^i_{\alpha}} \cap n_r = R_{\xi^i_{\alpha_0}} \cap n_r$ for each $\alpha \in I_0$ and $i < \ell$. Notice that $I_0 \in M$ since $I\in M$ and the refinement procedure is definable. Moreover, $|I_0| = \om_1$. Otherwise, $I_0 \sub M$ and, in particular, $\alpha_0 \in M$.

To choose $\alpha_1$, consider the increasing tower $\left\{\underline{R}_\alpha \setminus R_A \colon \alpha \in I_0 \right\} \in N$. This tower is not bounded by the set $R_B$, hence there exist some $\alpha'_1 \notin N$ and $n_1 > n_r$ such that $n_1 \in \underline{R}_{\alpha'_1} \setminus R_r$.
Define $I_1 = \left\{\alpha \in I_0 \colon n_1 \in \underline{R}_{\alpha}\right\} \in N$.
Since $N\prec H(\theta)$ and $\alpha'_1\notin N$, the set $I_1$ is uncountable. Pick any $\alpha_1 \in I_1 \cap N$ such that $\max A < \min F'_{\alpha_1}$.

Suppose that $\alpha_j, I_j \in N$ are defined for some $j<k$.
Put $Z = R_A \cup \bigcup \left\{\overline{R}_{\alpha_i} \colon i \leq j \right\}$.
Consider the tower
$\left\{\underline{R}_\alpha \setminus Z \colon \alpha \in I_j \right\} \in N$.
This tower is not bounded by $R_B$ hence there exist some 
$\alpha'_{j+1} \notin N$ and $n_{j+1} > n_r$ such that 
$n_{j+1} \in \underline{R}_{\alpha'_{j+1}} \setminus (Z \cup R_r)$.
Define \[I_{j+1} = \left\{\alpha \in I_j \colon n_{j+1} \in \underline{R}_{\alpha}\right\} \in N.\]
Again, since $N\prec H(\theta)$ and $\alpha'_{j+1}\notin N$, the set $I_{j+1}$ is uncountable.
Pick any $\alpha_{j+1} \in I_{j+1} \cap N$, $\alpha_{j+1} > \alpha_j$.
\end{proof}

Define the condition $s$ by $F_s = F_r \cup \bigcup_{j \leq k} F'_{\alpha_j}$, $n_s>\max\{n_i\colon i\leq k\}+n_r$,
\begin{itemize}
 \item $L^\xi_s = L^\xi_r$ for $\xi \in F_r$, $\xi \leq \max \Delta$;
 \item $L^\xi_s = L^\xi_r \cup \bigcup \left\{n_j \colon 0<j\leq k\right\}$ for $\xi \in F_r$, 
 $\max \Delta < \xi$;
 \item $L^{\xi^i_{\alpha_j}}_s = L^{\xi^i_{\alpha_0}}_r \cup \{n_{j+1}\}$ for $i<\ell$, $0<j< k$;
\item $L^{\xi^i_{\alpha_k}}_s = L^{\xi^i_{\alpha_0}}_r$ for $i<\ell$.
\end{itemize}
It is not difficult to verify that $s\in \mathbb{P}$ and $s<r$, $s<p_{\alpha_j}$ for each $j\leq k$.
Now \[ s \Vdash \left\{\alpha_0,\alpha_1,\ldots, \alpha_k, \beta\right\} \subseteq \dot{X}.\]
Moreover $L^{\alpha_j}_s \setminus L^\beta_s = \{n_0\}$ for each $j \leq k$, and so $s\Vdash L_{\alpha_j} \setminus L_\beta = \{n_0\}$.
This is a contradiction with $s<r$.
\end{proof}

\end{example}

Theorem \ref{David0} together with Proposition \ref{half-is-special} immediately gives us the corollary promised in the previous section:

\begin{cor}
There is an $\om_1$-tower equivalent neither to a Hausdorff nor to a Suslin tower in models obtained by adding $\om_1$ Cohen reals. 
\end{cor}

\begin{remark} Perhaps the left half of the gap constructed by Hirschorn in \cite{hirschorn} also has the above property. Hirschorn showed that in the model obtained by adding $\om_1$ random reals, one can generically add a gap $(L_\alpha, R_\alpha)_{\alpha<\om_1}$ which is left-oriented but not
Hausdorff. Hence $(L_\alpha)_{\alpha<\om_1}$ cannot be equivalent to a Suslin tower. To show that $(L_\alpha, R_\alpha)_{\alpha<\om_1}$ is not Hausdorff, Hirschorn used a certain fact based on Gilles theorem (\cite[Lemma 5.5]{hirschorn}). This fact
can be immediately modified for the case of towers in the following way. Assume that $(\mathcal{R},\lambda)$ is the random algebra with the standard measure and $(\dot{T}_\alpha)_{\alpha<\om_1}$ is an $\mathcal{R}$-name for a tower. If there is a function $h\colon \om\to \mathbb R^+$ converging to $0$ such that 
	\[ \lambda\left(\Arrowvert \dot{T}_\alpha \sub \dot{T}_\beta \cup n \Arrowvert\right) \leq h(n) \]
	for each $\alpha<\beta<\om_1$ and $n< \om$, then $(\dot{T}_\alpha)_{\alpha<\om_1}$ is not Hausdorff. (Here $\Arrowvert \varphi \Arrowvert$ represents the Boolean value of the sentence $\varphi$). However, it does not seem that
	$(L_\alpha)_{\alpha<\om_1}$ satisfies this condition for any $h\colon \om\to \mathbb R^+$ converging to $0$.
\end{remark}



\section{Towards a structure theory: Tukey order on towers}\label{Tukey}

Throughout this section we deal only with towers of length $\om_1$. As we have seen, we can single out several classes of towers defined by their ``inclusion structure''.  It is natural to ask if we can go further in this analysis. A research of this kind was done for
ultrafilters in \cite{Natasha}, using the classification of Tukey types.

We present here basic facts concerning Tukey order. See \cite{directed-stevo, Natasha} for more details and for the complete bibliography.

\begin{defi}
	Let $\mathcal{D}$ and $\mathcal{E}$ be directed sets. A function 
	$g \colon \mathcal{D} \to \mathcal{E}$ is \emph{Tukey} if the image of every unbounded subset of $\mathcal{D}$ is unbounded in $\mathcal{E}$.
	In such case, we say that $\mathcal{E}$ is \emph{Tukey above} $\mathcal{D}$ ($\mathcal{D} \leq_T \mathcal{E}$).
	If $\mathcal{D} \geq_T \mathcal{E} \geq_T \mathcal{D}$, then  $\mathcal{D}$ and $\mathcal{E}$ are said to be \emph{Tukey equivalent}, $\mathcal{D} \equiv \mathcal{E}$.
\end{defi}

\begin{prop} \label{basic-Tukey}
	If $\mathcal{D}$, $\mathcal{E}$ are directed posets such that $\mathcal{D}$ is a cofinal subset of $\mathcal{E}$, then $\mathcal{D}\equiv \mathcal{E}$.
\end{prop}

\begin{thm}(see \cite{directed-stevo}) \label{directed}
Let $D$ be a directed poset of size at most $\om_1$. Then either $D\equiv 1$, or $D\equiv \omega$, 
or $D\equiv \omega_1$, or $[\om_1]^{<\om} \geq_T D \geq_T \om\times\om_1$. Moreover, under $\mathsf{PFA}$ there are no Tukey types in between
$\om\times\om_1$ and $[\om_1]^{<\om}$.
\end{thm}

We have to agree on which emanations of towers we want to examine. Towers ordered by ``$\sub$'' are not satisfactory because we do not really want to pay attention to finite modifications of levels. It is also more convenient to deal with directed
sets. Structure theory for non-directed posets is available (see \cite{oriented-stevo}), but seems to be a bit cumbersome. The right structure to study seems to be the ideal generated by the tower (and all finite subsets of $\omega$). 
As before, we denote it by $\langle\mathcal{T}\rangle$ for a given tower $\mathcal{T}$, this time understanding it as the structure $(\langle\mathcal{T}\rangle, \sub)$.
The only inconvenience is that  $\langle\mathcal{T}\rangle$ has cardinality continuum. For this reason we also consider a cofinal directed subset of $(\langle\mathcal{T}\rangle,\sub)$ consisting of finite modifications of elements of $\mathcal T$,
 \[\langle\mathcal{T}\rangle_* = \{T \cup n \colon T \in \mathcal T, n \in \om\}.\]






\begin{defi}[\cite{directed-stevo}]
	Let $\mathcal D $ be a directed poset of cardinality $\om_1$. We say that $\mathcal D$ has property $(\dagger)$ if every 
	uncountable subset of $\mathcal D$ contains a countable unbounded subset.
\end{defi}

It is easy to see that if $\mathcal D$ has $(\dagger)$, then $\om\times\om_1 <_T \mathcal D$.

\begin{thm}[\cite{directed-stevo}]\label{MA-Tukey}
	Assume $\mathsf{MA}_{\om_1}$. If a directed poset $\mathcal D$ of cardinality $\om_1$ has $(\dagger)$, then $\mathcal D\equiv [\om_1]^{<\om}$.
\end{thm}

\begin{prop}\label{tower-dagger}
	The poset $\langle\mathcal T \rangle_*$ has property $(\dagger)$ for every tower $\mathcal T$.
\end{prop}

\begin{proof}
	Let $\mathcal S$ be an uncountable subset of $\langle\mathcal T \rangle_*$. We can assume that $\mathcal S$ is an
	increasing tower cofinal in $\mathcal T$, $\mathcal S = \{S_\alpha \colon \alpha < \om_1\}$. Suppose that for each $\beta<\om_1$ the set
	$\{S_\alpha \colon \alpha < \beta\}$ is bounded by an element of $\langle\mathcal S \rangle_*$. 
	In particular, this means that  $\left(\bigcup_{\alpha < \beta} S_\alpha\right)_{\beta<\om_1}$ does not stabilize.
	Hence $\left(\bigcup_{\alpha < \beta} S_\alpha\right)_{\beta<\omega_1}$ is an uncountable strictly increasing $\sub$-chain, a contradiction.
\end{proof}


Theorem \ref{MA-Tukey} now implies that under $\mathsf{MA}_{\omega_1}$ there is only one Tukey type of $\om_1$-towers.
\begin{cor} \label{MA-top} 
 Every ideal generated by a tower is Tukey top under $\mathsf{MA}_{\omega_1}$. 
\end{cor}

This should be contrasted with the following.

\begin{thm}
	Assuming $2^{\omega_1} > \omega_2$ and $\mathsf{CH}$. There are $2^{\mathfrak c}$ many incomparable Tukey classes represented by tower ideals.
\end{thm}
\begin{proof}
	According to \cite[Corollary 23]{Natasha}, if $2^{\omega_1} > \omega_2$, then there are $2^{\mathfrak c}$ many incomparable Tukey types of P-points.
Each P-point is generated by a tower filter (which is its cofinal subset). Now use Proposition \ref{basic-Tukey}.
\end{proof}

\begin{thm}\label{Hausdorff=top}
	A tower $\mathcal T$ is Hausdorff iff $\langle\mathcal T\rangle  \equiv [\om_1]^{<\om}$.
\end{thm}
\begin{proof}
Let $\mathcal H$ be a cofinal subtower of $\mathcal T$ satisfying (H).
We show that each infinite subset of $\mathcal H$ is unbounded in $\langle\mathcal T\rangle_*$ 
(and hence any injective map from $[\om_1]^{<\om}$ into $\mathcal H$ is a Tukey function from $[\om_1]^{<\om}$ to $\langle\mathcal T\rangle_* $).
Pick any countable set $A = \{T_\alpha \colon \alpha \in I\} \subseteq \mathcal H$ 
and suppose that $X \in \langle\mathcal T\rangle_*$ is an upper bound of $A$.
There is some $T_\beta \in \mathcal H$, $\sup I < \beta$, and $n < \om$ such that $X \subseteq T_\beta \cup n$.
The set $\{\alpha \in I \colon T_\alpha \subseteq T_\beta \cup n\}$ is finite since $\mathcal H$ satisfies (H).
Thus there is some $\alpha \in I$ such that $T_\alpha \nsubseteq T_\beta \cup n$, and hence $T_\alpha \nsubseteq X$. A contradiction.

For the other direction, consider a Tukey map $f\colon [\om_1]^{<\om} \to \langle\mathcal T\rangle_*$.
We may suppose without loss of generality that $f(\{\beta\}) \setminus f(\{\alpha\})$ is infinite iff $\alpha < \beta < \om_1$.
We show that the tower $\mathcal S = (f(\{\alpha\}))_{\alpha<\om_1}$ satisfies condition (H).
Suppose that for some $\beta < \om_1$ and $n < \om$ the set 
$A = \{\alpha<\beta \colon f(\{\alpha\}) \setminus f(\{\beta\}) \subseteq n \}$ is infinite.
Then $\{f(\{\alpha\})\colon \alpha\in A\}$ is bounded by $f(\{\beta\})\cup n$ in $\langle\mathcal T\rangle_*$, a contradiction.
Notice that the towers $\mathcal T$ and $\mathcal S$ generate the same ideal, so $\mathcal T$ is Hausdorff.
\end{proof}

\begin{prop}
Consistently, there are Suslin towers $\mathcal{T}^0$, $\mathcal{T}^1$ such that $\langle\mathcal{T}^0 \rangle \times \langle \mathcal{T}^1\rangle$ is Tukey top.
\end{prop}
\begin{proof}
Consider a Hausdorff tower $\mathcal T = \left(T_\alpha\right)_{\alpha<\om_1}$. 
Then in a model obtained by adding a Cohen real $c \subseteq \om$ define $T^0_\alpha = T_\alpha \cap c$ and $T^1_\alpha = T_\alpha \setminus c$. 
By Proposition~\ref{cohen-inclusion-tower} both of these towers are Suslin. 
The map $f \colon \langle \mathcal T \rangle_* \to \langle\mathcal{T}^0 \rangle_* \times \langle \mathcal{T}^1\rangle_*$ defined by 
$T_\alpha \cup n \mapsto \left(\left( T_\alpha \cap c \right)\cup n, \left( T_\alpha \setminus c \right) \cup n \right)$
for $\alpha < \om_1, n < \om$ is Tukey, so $\langle\mathcal{T}^0 \rangle \times \langle \mathcal{T}^1\rangle \equiv \langle \mathcal T \rangle$.
\end{proof}

We do not know if the statement of the above proposition is true whenever there is a Suslin $\om_1$-tower.

Notice that by putting together Theorem~\ref{Hausdorff=top} and Corollary~\ref{MA-top}, 
we obtain an alternative proof of Corollary \ref{MA-Hausdorff}.

\begin{cor}
	If $\mathcal{T}$ is a Suslin tower, then $\langle\mathcal{T}\rangle <_T [\om_1]^{<\om}$.
\end{cor}

The last fact is of course an immediate consequence of Theorem \ref{Hausdorff=top}, but it can be proved directly using the fact 
that each uncountable subtower of a Suslin tower contains a $\sub$-chain of order type $\om+1$. Indeed, by the Erd\"os-Dushnik-Miller theorem 
($\om_1 \rightarrow (\om_1, \om+1)^2$)  (see \cite[Theorem 11.3]{erdos}) 
we know that either there is an uncountable $\sub$-antichain in
the subtower or a $\sub$-chain of length $\om+1$. The first alternative is clearly not possible. It follows that uncountable well-ordered subsets of the ideals generated by Suslin towers have infinite bounded subsets, so they cannot be Tukey equivalent to $[\om_1]^{<\om}$.

Theorem \ref{Hausdorff=top} gives us one more useful information along these lines. It is not easy to point out the reason why a given tower is not Hausdorff other than the lack of uncountable $\sub$-antichains. Consider the following property of a 
tower $(T_\alpha)_{\alpha<\om_1}$: \emph{for
every uncountable $X\sub \om_1$ there is an infinite $I\sub X$ and $\alpha>\sup I$ such that $\bigcup_{\xi\in I}T_\xi \sub^* T_\alpha$}. By Theorem \ref{Hausdorff=top}, this property is equivalent to saying that $(T_\alpha)_{\alpha<\om_1}$ is not
Hausdorff.


Tukey theory harmonizes with the intuition that the Hausdorff property is in a sense more important than the property of possessing an uncountable 
$\sub$-antichain. It is not clear for us if there are other critical Tukey types of tower ideals.

\begin{remark}
The above approach has a disadvantage. Generating an ideal can loose the information if the generating tower is Suslin. Instead of examining ideals generated by towers, one can investigate the structure $\langle \{T =^* T_\alpha \colon
\alpha < \om_1 \},\sub, \cup,\cap \rangle$ for a given tower $(T_\alpha)_{\alpha<\om_1}$. It is easy to see that being Suslin is invariant under isomorphism of such lattices. 
\end{remark}

\section{Questions}

In this section we list some questions and open problems related to the topic of this paper.

\begin{prob}
Is it consistent that each Hausdorff tower is the left half of a Hausdorff gap?
\end{prob}

Notice that the standard Hausdorff construction (of a Hausdorff gap) cannot be modified in an obvious way to produce a Hausdorff tower without creating the other half of a Hausdorff gap as a byproduct. In Section \ref{basic} we showed a consistent
example of a Hausdorff tower which is maximal (see Remark \ref{maximal-Hausdorff}) and hence is not a half of any gap.  

\begin{prob}
Is it consistent that all $\om_1$-towers/gaps are special but there is a non-Hausdorff tower/gap?
\end{prob}

In particular, we can ask the following:

\begin{prob}
Does $\mathsf{OCA}$ implies that every $\om_1$-tower/gap is Hausdorff?
\end{prob}

The natural attempt to answer this question in negative would be to start with a model with a special non-Hausdorff tower/gap and show 
that forcing $\mathsf{OCA}$ preserves its non-Hausdorffness. 

Every $\sub^*$-descending tower generates a filter in $\pofin$, a closed subset of the space of ultrafilters $\om^*$. It is natural
to ask if the closed sets generated by Hausdorff towers possess some special properties.

\begin{prob}
Is there some characterization of the Hausdorff property of towers in topological terms? 
\end{prob}

Perhaps the next question can be solved using coherent sequences. They produce towers in a nice way, but it is not 
clear how to analyze the properties of resulting towers.

\begin{prob}
Does $\mathfrak{t}=\om_{1}$ imply that there is a maximal Hausdorff tower?
\end{prob}

Since each Hausdorff tower generates a meager ideal, a positive answer would provide a dense meager $\om_1$-generated P-ideal.
Example 1 in \cite{oscillations} shows that existence of these objects is in fact equivalent. 

\begin{prob}\label{tall&meager}
Is there a model in which every ideal generated by an ($\om_1$-)tower is dense only if it is non-meager?
\end{prob}

Note that this problem for $\om_1$-towers is interesting only if we add the requirement $\mathfrak t = \om_1$.
The conjecture here is that there is no such model, i.e.\ a meager dense $\om_1$-generated P-ideal should be constructible from the assumption $\mathfrak t = \om_1$.
Obviously, if $\om_1 < \mathfrak b$, every $\om_1$-generated ideal is meager.
If $\non(\mathcal{\mathcal M}) = \om_1$ ($\mathcal M$ is the~ideal of meager subsets of $2^\om$), then there is such an ideal by the following argument due to M.\,Hru\v{s}\'ak.

For a tall ideal $\mathcal I \subseteq \Pw{\omega}$ define
\[\cov^*(\mathcal I) = \min\{|\mathcal A |\colon \mathcal A \subset \mathcal I, (\forall X \in [\om]^\om )(\exists A \in\mathcal A)(|A \cap X| = \om ) \}.\]
It follows from \cite[Propositions 1.5, 3.1, 3.2]{Hrusak-Hernandez} that 
$\cov^*(\mathcal I ) \leq \non(\mathcal M)$ for each tall analytic P-ideal $\mathcal I$.
Thus if $\non(\mathcal{\mathcal M}) = \om_1$, for any given tall analytic P-ideal $\mathcal I$ there is a tall $\om_1$-tower which generates an ideal contained in $\mathcal I$, hence is meager.


A gap $\left(f_\alpha, g_\alpha \right)_{\alpha < \om_1}$ in $(\fc,<^*)$ is \emph{tight} if
$\left(f_\alpha\restriction A, g_\alpha\restriction A \right)_{\alpha < \om_1}$
is a gap in $(\lxp{\om}{A},<^*)$ for each infinite $A \sub \om$.
A positive answer to the following problem would provide a~negative answer to Problem~\ref{tall&meager}.

\begin{prob}
	Is the assumption $\mathfrak t = \om_1$ equivalent with the existence of a tight gap in~$(\fc,<^*)$.
\end{prob}

In connection with the previous problem, let us mention that the Borel weak diamond principle $\diamondsuit(2,=)$ of \cite{parametr_diamonds} implies the existence of
a tight gap (this was suggested by M.\,Hru\v{s}\'ak). In fact, it even implies the existence of a peculiar gap (see~\cite{shelah885} for definition). 
Note also that there are no peculiar gaps in the model from~\cite{Dow-Steprans}, but $\mathfrak b = \om_1$.

We know that there can be Suslin towers generating meager ideals. However, it is unclear whether they are not equivalent to special towers.

\begin{prob}
Is there a tower generating a tall meager ideal which is not equivalent to a special tower? 
\end{prob}

In Section \ref{Tukey} we mentioned that in each Suslin tower there is a $\sub$-chain of order type $\om+1$. It seems natural to ask the following:

\begin{prob}
How long $\sub$-chains have to exist in Suslin towers?
\end{prob}

\section{Acknowledgments}

Most of the results presented in this article were obtained during the stay of the authors at Fields Institute in Toronto.  We would hereby like to thank the institute for the excellent working conditions. The first author wants to acknowledge the
support of the INFTY Research Networking Program (European Science Foundation) via grant~4955. We are grateful to several mathematicians, who agreed to devote their time to discuss the subject of this paper, including Grzegorz Plebanek, Michael Hru\v{s}\'ak, 
J\"{o}rg Brendle, and Teruyuki Yorioka. We are particularly indebted to Stevo Todor\v{c}evi\'c for all his invaluable remarks and insightful suggestions.
We would also like to thank the anonymous referee for the careful reading of the article and for dozens of valuable comments.


\bibliographystyle{alpha}
\bibliography{hausdorff}
\end{document}